\documentclass{amsart}
\usepackage{amssymb}

\newtheorem{theorem}{Theorem}[section]
\newtheorem{lemma}[theorem]{Lemma}
\newtheorem{proposition}[theorem]{Proposition}
\newtheorem{remark}[theorem]{Remark}
\newtheorem{corollary}[theorem]{Corollary}

\numberwithin{equation}{section}
\begin{document}

\title[Combinatorial bases of
$B_2\sp{(1)}$ modules]{Combinatorial bases of modules for affine Lie
algebra $B_2\sp{(1)}$}
\author{Mirko Primc}
\address{Department of Mathematics, University of Zagreb, Bijeni\v{c}ka 30, Zagreb, Croatia}
\email{primc@math.hr}
\thanks{Partially supported  by the Ministry of Science and Technology of the
Republic of Croatia, grant 037-0372794-2806.}

\begin{abstract}
In this paper we construct bases of standard (i.e. integrable
highest weight) modules $L(\Lambda)$ for affine Lie algebra of type
$B_2\sp{(1)}$ consisting of semi-infinite monomials. The main
technical ingredient is a construction of monomial bases for
Feigin-Stoyanovsky type subspaces $W(\Lambda)$ of $L(\Lambda)$ by
using simple currents and intertwining operators in vertex operator
algebra theory. By coincidence $W(k\Lambda_0)$ for $B_2\sp{(1)}$ and
the integrable highest weight module $L(k\Lambda_0)$ for
$A_1\sp{(1)}$ have the same parametrization of combinatorial bases
and the same presentation $\mathcal P/\mathcal I$\,.
\end{abstract}
\maketitle

\section{Introduction}
B.L.~Feigin and A.V.~Stoyanovsky gave in \cite{FS} a construction of
bases of standard (i.e. integrable highest weight) modules
$L(\Lambda)$ for affine Lie algebra $\hat{\mathfrak g}$ of type
$A_1\sp{(1)}$ consisting of semi-infinite monomials. In \cite{P1}
such a construction is extended to all standard modules for affine
Lie algebras of type $A_n\sp{(1)}$. The construction starts by
choosing a particular $\mathbb Z$-grading of the corresponding
simple Lie algebra
\begin{equation}\label{E:general Z-grading}
\mathfrak g={\mathfrak g}_{-1}+{\mathfrak g}_{0}+{\mathfrak g}_{1}
\end{equation}
and a particular group element $e$ which normalizes the subalgebra
$\hat{\mathfrak g}_{1}={\mathfrak g}_{1}\otimes \mathbb C
[t,t\sp{-1}]$. Then
$$
L(\Lambda)=\bigcup_{m=0}\sp{\infty}\,e\sp{-m}U(\hat{\mathfrak
g}_{1})v_\Lambda,\qquad e\sp{-m-1}U(\hat{\mathfrak
g}_{1})v_\Lambda\supset e\sp{-m}U(\hat{\mathfrak g}_{1})v_\Lambda
$$
and semi-infinite monomials appear by ``taking a limit''
$$
\lim_{m\to\infty}e\sp{-m}U(\hat{\mathfrak g}_{1})v_\Lambda.
$$
On the other side, for any classical simple Lie algebra $\mathfrak
g$ and any possible $\mathbb Z$-grading \eqref{E:general Z-grading}
such construction is given in \cite{P2} for the basic
$\hat{\mathfrak g}$-module $L(\Lambda_0)$. In each of these cases a
weight basis of ${\mathfrak g}_{1}$ is interpreted as a perfect
crystal for the quantum group $U_q({\hat{\mathfrak g}}_0)$ and in a
proof of linear independence a crystal base character formula
\cite{KKMMNN} is used, but it was not clear ``why'' this proof works
and how such approach could be extended to higher level standard
modules. A new understanding came from the works of G.~Georgiev
\cite{G} and S.~Capparelli, J.~Lepowsky and A.~Milas \cite{CLM1} and
\cite{CLM2} based on a general idea of J.~Lepowsky to use
intertwining vertex operators to build bases of standard modules and
obtain Rogers-Ramanujan-type recursions for their graded dimensions.
Their way of using intertwining operators inspired a simpler proof
of linear independence for $A_n\sp{(1)}$ in \cite{P3}, and new
constructions for $D_4\sp{(1)}$ by I.~Baranovi\' c in \cite{Ba} and
for $A_n\sp{(1)}$ by G.~Trup\v cevi\' c in \cite{T} for all possible
$\mathbb Z$-gradings \eqref{E:general Z-grading}.

In this paper we use Capparelli-Lepowsky-Milas' approach to extend
the construction in \cite{P2} to all standard modules $L(\Lambda)$
for affine Lie algebra $\hat{\mathfrak g}$ of type $B_2\sp{(1)}$. In
this case we neither have a lattice construction of level $1$
modules nor Dong-Lepowsky's intertwining operators \cite{DL}, but we
manage to construct intertwining operators we need in a proof of
linear independence by using vertex operator algebra theory and
results of C.~Dong, H.~Li and G.~Mason \cite{DLM} on simple
currents. Along the way we also obtain a presentation theorem for
Feigin-Stoyanovsky type subspaces. The underlying structure of
Feigin-Stoyanovsky type subspaces is parallel to the structure of
principal subspaces studied, for example, in \cite{G}, \cite{Cal},
\cite{CalLM} and \cite{AKS}.

Since the list of all possible $\mathbb Z$-gradings \eqref{E:general
Z-grading} coincides with the list of all possible level $1$ simple
currents constructed in \cite{DLM}, the results and methods used in
\cite{T}, \cite{Ba} and this paper give hope that the construction
in \cite{P2} might be extended to all standard modules of all
classical affine Lie algebras by using intertwining operators. In
return, one should expect a rich and interesting combinatorial
structure behind this construction, on one side extending
combinatorics of infinite paths used in \cite{KKMMNN}, and on the
other side extending $(k,n+1)$-admissible configurations ---
combinatorial objects introduced and studied in a series of papers
\cite{FJLMM}--\cite{FJMMT}. Moreover, it might be that the reason
``why'' the proof in \cite{P2} works can be explained by some
connection of tensor multiplication of vertex operator algebra
modules with simple currents (cf. \cite{HL}, \cite{Fu}, \cite{DLM})
on one side and tensor multiplication of affine crystals with
perfect crystals (cf. \cite{KKMMNN}) on the other.

Let $\mathfrak g$ be a simple complex Lie algebra of type $B_2$, let
$\mathfrak h$ be a Cartan subalgebra of $\mathfrak g$ and
\begin{equation*}
\mathfrak g={\mathfrak g}_{-1}+{\mathfrak g}_{0}+{\mathfrak g}_{1}
\end{equation*}
a $\mathbb Z$-grading of $\mathfrak g$ such that ${\mathfrak
h}\subset{\mathfrak g}_{0}$. We fix a basis of ${\mathfrak g}_{1}$
consisting of root vectors denoted as
$$
x_{\underline{2}}\,, x_0\,, x_2\,.
$$
Let $\hat{\mathfrak g} = {\mathfrak g}\otimes \mathbb C
[t,t\sp{-1}]+\mathbb C c+\mathbb C d$ be the associated affine Lie
algebra with the canonical central element $c$. For $x\in\mathfrak
g$ and $n\in\mathbb Z$ we write $x(n)=x\otimes t\sp n$. Then for
integral dominant weight
$$\Lambda=k_0\Lambda_0+k_1\Lambda_1+k_2\Lambda_2$$ of level
$k=k_0+k_1+k_2$ a basis of standard module $L(\Lambda)$ can be
parametrized by semi-infinite monomials
\begin{equation}\label{E:introduction monomials}
\prod_{j\in\mathbb
Z}x_{\underline{2}}(-j)\sp{c_j}x_0(-j)\sp{b_j}x_2(-j)\sp{a_j}\,,\qquad
c_j=b_j=a_j=0\quad\text{for}\quad -j\ll 0,
\end{equation}
with quasi-periodic tail with the period of length $6$
\begin{equation*}
\begin{aligned}
(\dots, \ &c_{-2n},\ b_{-2n},\ a_{-2n},\ c_{-2n-1},\ b_{-2n-1},\
a_{-2n-1},\ \dots)\\=(\dots,& \ k_1,\ k_2,\ k_1,\ k_0,\ k_2,\ k_0,\
\dots)
\end{aligned}
\end{equation*}
for $n\gg 0$, satisfying for all $j\in\mathbb Z$ the so called
difference conditions
\begin{equation}\label{E:introduction difference conditions}
\begin{aligned}
c_{j+1}+b_{j+1}+c_{j}&\leq k,\\
b_{j+1}+a_{j+1}+c_{j}&\leq k,\\
a_{j+1}+c_{j}+b_{j}&\leq k,\\
a_{j+1}+b_{j}+a_{j}&\leq k.
\end{aligned}
\end{equation}
This is the Corollary~\ref{C: bases as semi-infinite monomials} of
Theorem~\ref{T:semi infinite monomials basis}. The main technical
ingredient in the proof is a construction of monomial bases for
Feigin-Stoyanovsky type subspaces defined as
$$
W(\Lambda)=U(\hat{\mathfrak g}_{1})v_\Lambda\subset L(\Lambda),
$$
where $\hat{\mathfrak g}_{1}={\mathfrak g}_{1}\otimes \mathbb C
[t,t\sp{-1}]$ and $v_\Lambda$ is a highest weight vector in
$L(\Lambda)$. By Theorem~\ref{T:the main theorem} the constructed
basis for level $k$ subspace $W(\Lambda)$ consists of finite
monomials of the form \eqref{E:introduction monomials} with
$-j\leq-1$, satisfying difference conditions \eqref{E:introduction
difference conditions} and the so called initial conditions
$$
a_1\leq k_0,\quad b_1+a_{1}\leq k_0+k_2\quad\text{and}\quad
c_1+b_{1}\leq k_0+k_2.
$$
Another consequence of this result is Theorem~\ref{T:presentation of
W(Lambda)} which gives a presentation
$$
W(\Lambda)\cong \mathcal P/\mathcal I_\Lambda,
$$
where $\mathcal P$ is a polynomial algebra $\mathbb
C[x_{\underline{2}}(j),x_0(j),x_2(j)\mid j\leq-1]$ and $\mathcal
I_\Lambda$ is the ideal generated by the set of polynomials
$$\begin{aligned}
&\bigcup_{\substack{n\leq-k-1}}U(\mathfrak
g_0)\cdot\Big(\sum_{\substack{j_1,\dots,j_{k+1}\leq-1\\j_1+\dots+j_{k+1}=n}}x_2(j_1)\dots
x_2(j_{k+1})\Big)\\
&\quad\bigcup \{x_2(-1)\sp{k_0+1}\}\bigcup U(\mathfrak g_0)\cdot
x_2(-1)\sp{k_0+k_2+1}.
\end{aligned}
$$

By coincidence $W(k\Lambda_0)$ for $B_2\sp{(1)}$ and the integrable
highest weight module $L(k\Lambda_0)$ for $A_1\sp{(1)}$ have the
same parametrization of combinatorial bases and the same
presentation $\mathcal P/\mathcal I$\,. Due to this coincidence
E.~Feigin's fermionic formula \cite{F} for $A_1\sp{(1)}$-module
$L(k\Lambda_0)$ is also a character formula of Feigin-Stoyanovsky
type subspace $W(k\Lambda_0)$ for $B_2\sp{(1)}$.

As it was already said, in our construction we use simple currents
and intertwining operators for vertex operator algebra
$L(\Lambda_0)$ associated with the affine Lie algebra
$\hat{\mathfrak g}$ at level $1$. To be more precise, we use results
in \cite{DLM} and \cite{L2} to see the existence of level $1$
``simple current operators''
\begin{equation*}
L(\Lambda_0)\overset{[\omega]}\longrightarrow
L(\Lambda_{1})\overset{[\omega]}\longrightarrow
L(\Lambda_{0}),\qquad L(\Lambda_{2})
\overset{[\omega]}\longrightarrow L(\Lambda_{2})
\end{equation*}
which are linear bijections with the crucial property
\begin{equation}\label{E: crucial property of omega}
x(n)[\omega]=[\omega]x(n+1) \quad \text{for all}\quad
x(n)\in\hat{\mathfrak g}_1.
\end{equation}
From \cite{L1} we have fusion rules
$$
\dim I\binom{L(\Lambda_2)}{L(\Lambda_2)\quad
L(\Lambda_0)}=1\qquad\text{and}\qquad \dim
I\binom{L(\Lambda_1)}{L(\Lambda_2)\quad L(\Lambda_2)}=1
$$
from which we deduce that there are coefficients $[\omega_2]$ and
$[\omega_{\underline{2}}]$ of intertwining operators
$$
\begin{aligned}
&L(\Lambda_0)\overset{[\omega_2]}\longrightarrow
L(\Lambda_{2})\overset{[\omega_{\underline{2}}]}\longrightarrow
L(\Lambda_{1}),\quad
v_{\Lambda_0}\overset{[\omega_2]}\longrightarrow
v_{\Lambda_{2}}\overset{[\omega_{\underline{2}}]}\longrightarrow
v_{\Lambda_{1}},\quad [\omega_{\underline{2}}]w_{\underline{2}}=0,\\
&L(\Lambda_0)\overset{[\omega_{\underline{2}}]}\longrightarrow
L(\Lambda_{2})\overset{[\omega_{2}]}\longrightarrow L(\Lambda_{1}),
\quad v_{\Lambda_0}\overset{[\omega_{\underline{2}}]}\longrightarrow
w_{\underline{2}}\overset{[\omega_2]}\longrightarrow
v_{\Lambda_{1}},\quad [\omega_{2}]v_{\Lambda_{2}}=0
\end{aligned}
$$
which commute with the action of $\hat{\mathfrak g}_1$. We consider
higher level standard modules as submodules of tensor products of
level $1$ modules
$$
L(\Lambda)\subset L(\Lambda_0)\sp{\otimes k_0} \otimes
L(\Lambda_{1})\sp{\otimes k_1}\otimes L(\Lambda_{2})\sp{\otimes
k_2}.
$$

Behind all combinatorial properties of our construction seems to be
relation \eqref{E:v(Lambda to Lambda star)} for $[\omega]v_\Lambda$,
written in terms of tensor products of level $1$ highest weight
vectors as
\begin{equation}\label{E:introduction v(Lambda to Lambda star)}
\begin{aligned}
&[\omega]\left(v_{\Lambda_0}\sp{\otimes k_0}\otimes
v_{\Lambda_1}\sp{\otimes k_1}\otimes v_{\Lambda_{2}}\sp{\otimes
k_2}\right)\\
&=\big([\omega]v_{\Lambda_0}\big)\sp{\otimes k_0}\otimes
\big([\omega]v_{\Lambda_1}\big)\sp{\otimes k_1}\otimes
\big([\omega]v_{\Lambda_2}\big)\sp{\otimes
k_2}\\
&=C\,x_{\underline{2}}(-1)\sp{k_1}x_0(-1)\sp{k_2}x_2(-1)\sp{k_1}\left(v_{\Lambda_1}\sp{\otimes
k_0}\otimes v_{\Lambda_0}\sp{\otimes k_1}\otimes
v_{\Lambda_{2}}\sp{\otimes k_2}\right).
\end{aligned}
\end{equation}
In particular, it is this relation that for level $1$ modules makes
the use of crystal base character formula \cite{KKMMNN} in \cite{P2}
possible.

Very roughly speaking, we prove linear independence by induction on
degree of basis elements in two steps: for monomial vectors
$x(\pi)v_{\Lambda}$, which appear with nontrivial  coefficients
$c_\pi\neq0$ in a linear combination $\sum c_\pi
x(\pi)v_{\Lambda}=0$, we first use intertwining operators
$x(\pi)v_{\Lambda}\to x(\pi)v_{\Lambda'}$ to be able to apply
formula \eqref{E:introduction v(Lambda to Lambda star)} to vectors
$x(\pi)v_{\Lambda'}$ and get a combination of monomial vectors of
the form $\sum c_\pi x(\pi')[\omega]v_{\Lambda''}$. Then, as a
second step, we commute $[\omega]$ to the left and, by using
\eqref{E: crucial property of omega} and induction hypothesis, we
get that $c_\pi$ equals zero --- a contradiction. Of course, the
actual argument is a bit more complicated and, as in \cite{Ba}, we
have to use two basis elements of $4$-dimensional spinor $\mathfrak
g$-module on the top of $L(\Lambda_2)$ and the corresponding
coefficients $[\omega_2]$ and $[\omega_{\underline{2}}]$ of
intertwining operators.

A part of this paper was written while I was a member of the Erwin
Schr\" odinger Institute in Vienna in February of 2009. I would like
to thank J. Schwermer for his hospitality.

\section{Affine Lie algebra of type $B_2\sp{(1)}$}

Let ${\mathfrak g}$ be a complex simple Lie algebra of type $B_2$
and let $\mathfrak h$ be a Cartan subalgebra of ${\mathfrak g}$. Let
$\mathfrak g = \mathfrak h + \sum \mathfrak g_\alpha$ be a root
space decomposition of $\mathfrak g$. The corresponding root system
$R$ may be realized in $\mathbb R\sp2$ with the canonical basis
$\varepsilon_1, \varepsilon_2$ as
$$
R=\{\pm(\varepsilon_1-\varepsilon_2),\pm(\varepsilon_1+\varepsilon_2)\}\cup
\{\pm\varepsilon_1, \pm\varepsilon_2\}.
$$
We fix simple roots $\alpha_1=\varepsilon_1-\varepsilon_2$ and
$\alpha_2=\varepsilon_2$ and denote by $\omega_1=\varepsilon_1$ and
$\omega_2=\tfrac12(\varepsilon_1+\varepsilon_2)$ the corresponding
fundamental weights. Note that $\theta=\varepsilon_1+\varepsilon_2$
is the maximal root. Set
$$
\Gamma=\{\varepsilon_1-\varepsilon_2, \varepsilon_1,
\varepsilon_1+\varepsilon_2\}.
$$
Denote by $\langle\cdot,\cdot\rangle$ the normalized Killing form
such that $\langle\theta,\theta\rangle=2$. We identify $\mathfrak
h\cong \mathfrak h\sp*$ via $\langle\cdot,\cdot\rangle$. We fix
$$
\omega=\omega_1=\varepsilon_1.
$$
Then we have $\alpha(\omega)=\langle\alpha, \omega\rangle$ and
$$
\Gamma=\{\alpha\in R\mid \alpha(\omega)=1\}.
$$
Obviously we have a $\mathbb Z$-grading $\mathfrak g={\mathfrak
g}_{-1}+{\mathfrak g}_{0}+{\mathfrak g}_{1}$ for
$$
{\mathfrak g}_{0}=\mathfrak h+\sum_{\alpha(\omega)=0}\mathfrak
g_\alpha=\mathfrak h+\mathbb Cx_{\varepsilon_2}+\mathbb
Cx_{-\varepsilon_2},\qquad {\mathfrak g}_{\pm 1}=\sum_{\alpha\in\pm
\Gamma}\mathfrak g_\alpha.
$$
Clearly $\mathfrak g_1$ is an irreducible $3$-dimensional $\mathfrak
g_0$-module. We shall briefly write
$$
\underline{2}=\varepsilon_1-\varepsilon_2,
\quad0=\varepsilon_1,\quad 2=\varepsilon_1+\varepsilon_2
$$
so that $\Gamma=\{\underline{2}, 0, 2\}$, a notation as in \cite{P2}
and \cite{Ba}.
 For each root
$\alpha$ fix a root vector $x_\alpha$. For
$\alpha=\underline{2},0,2$ we shall write respectively $x_\alpha$ as
$$
 x_{\underline{2}},x_0, x_2.
$$
These vectors form a basis of $\mathfrak g_0$-module $\mathfrak
g_1$.

Denote by $\hat{\mathfrak g}$ the affine Lie algebra of type
$B_2\sp{(1)}$ associated to $\mathfrak g$,
$$
\hat{\mathfrak g} =\sum_{n\in\mathbb Z}{\mathfrak g}\otimes
t^{n}+\mathbb C c+\mathbb C d
$$
with the canonical central element $c$ and the degree element $d$
such that $[d,x\otimes t^{n}]=n\,x\otimes t^{n}$. Set
$$
\hat{\mathfrak g}_{< 0} =\sum_{n<0}{\mathfrak g}\otimes t^{n},\qquad
\hat{\mathfrak g}_{\leq 0} =\sum_{n\leq 0}{\mathfrak g}\otimes
t^{n}+\mathbb C c+\mathbb C d.
$$
Let $\alpha_0$, $\alpha_1$ and $\alpha_2$ be simple roots of
$\hat{\mathfrak g}$ with the root subspaces ${\mathfrak
g}_{-\theta}\otimes t\sp1$, ${\mathfrak g}_{\alpha_1}\otimes t\sp0$
and ${\mathfrak g}_{\alpha_2}\otimes t\sp0$ respectively, and let
$\Lambda_0$, $\Lambda_1$ and $\Lambda_2$ be the corresponding
fundamental weights of $\hat{\mathfrak g}$ (cf. \cite{K}). We write
$$
x(n)=x\otimes t^{n}
$$
for $x\in{\mathfrak g}$ and $n\in\mathbb Z$ and denote by
$x(z)=\sum_{n\in\mathbb Z} x(n) z^{-n-1}$  a formal Laurent series
in formal variable $z$. For
$$
\hat{\mathfrak g}_{0} =\sum_{n\in\mathbb Z}{\mathfrak g}_{0}\otimes
t^{n}+\mathbb C c+\mathbb C d,\qquad \hat{\mathfrak g}_{\pm1}
=\sum_{n\in\mathbb Z}{\mathfrak g}_{\pm1}\otimes t^{n}
$$
we have $\mathbb Z$-grading $\hat{\mathfrak g}=\hat{\mathfrak
g}_{-1}+\hat{\mathfrak g}_{0}+\hat{\mathfrak g}_{1}$. In particular,
${\hat{{\mathfrak g}}_1}$ is a commutative Lie subalgebra of
$\hat{{\mathfrak g}}$ with a basis
$$
\widehat\Gamma=\{x_{\underline{2}}(n), x_0(n), x_2(n)\mid
n\in\mathbb Z \}=\{x_\gamma(n)\mid \gamma \in \Gamma, n\in\mathbb Z
\}.
$$
On $\widehat\Gamma$ we use linear order
$$
\dots \prec x_2(n-1)\prec  x_{\underline{2}}(n)\prec  x_0(n)\prec
x_2(n)\prec  x_{\underline{2}}(n+1)\prec  \dots \,.
$$

\section{Feigin-Stoyanovsky type subspaces $W(\Lambda)$}

Denote by $L(\Lambda)$ a standard (i.e. integrable highest weight)
$\hat{\mathfrak g}$-module with a dominant integral highest weight
$$
\Lambda=k_0\Lambda_0+k_1\Lambda_1+k_2\Lambda_2,
$$
$k_0, k_1, k_2\in\mathbb Z_{+}$. Throughout the paper we denote by
$k=\Lambda(c)$ the level of $\hat{\mathfrak g}$-module $L(\Lambda)$,
$$
k=k_0 +k_1+ k_2
$$
(cf. \cite{K}). For each fundamental $\hat{\mathfrak g}$-module
$L(\Lambda_i)$ fix a highest weight vector $v_{\Lambda_i}$. By
complete reducibility of tensor products of standard modules, for
level $k>1$ we have
$$
L(\Lambda)\subset L(\Lambda_0)\sp{\otimes k_0} \otimes
L(\Lambda_{1})\sp{\otimes k_1}\otimes L(\Lambda_{2})\sp{\otimes k_2}
$$
with a highest weight vector
$$
v_{\Lambda} = v_{\Lambda_0}\sp{\otimes k_0}\otimes
v_{\Lambda_{1}}\sp{\otimes k_1}\otimes v_{\Lambda_{2}}\sp{\otimes
k_2}.
$$
Later on we shall also realize $L(\Lambda)$ in a symmetric algebra
$$
L(\Lambda)\subset S\sp k(L(\Lambda_0)\oplus L(\Lambda_1)\oplus
L(\Lambda_2)),\qquad v_{\Lambda} = v_{\Lambda_0}\sp{ k_0}
v_{\Lambda_{1}}\sp{ k_1} v_{\Lambda_{2}}\sp{ k_2}.
$$

We set $d\,v_\Lambda=0$. Then $L(\Lambda)$ is $\mathbb Z$-graded by
the degree operator $d$,
$$
L(\Lambda)=L(\Lambda)_0+L(\Lambda)_{-1}+L(\Lambda)_{-2}+\dots\,,
$$
and we say that $\mathfrak g$-module $L(\Lambda)_0=U(\mathfrak
g)v_\Lambda$ is the ``top'' of $L(\Lambda)$. The top of
$L(\Lambda_0)$ is trivial $\mathfrak g$-module $\mathbb
Cv_{\Lambda_0}$, the top of $L(\Lambda_1)$ is $5$-dimensional vector
representation $L(\omega_1)$ and the top of $L(\Lambda_2)$ is
$4$-dimensional spinor $\mathfrak g$-module $L(\omega_2)$.

For each integral dominant $\Lambda$ we have a Feigin-Stoyanovsky
type subspace
$$
W(\Lambda)=U({\hat{{\mathfrak g}}_1})v_\Lambda\subset L(\Lambda).
$$

Denote by $\pi\colon \{x_\gamma(-j)\mid \gamma \in \Gamma, j\geq
1\}\to \mathbb Z_+$ \ a ``colored partition'' for which a finite
number of ``parts'' $x_\gamma(-j)$ (of degree $j$ and color
$\gamma$) appear $\pi(x_\gamma(-j))$ times, and denote by
$$
x(\pi)=\prod x_\gamma(-j)\sp{\pi(x_\gamma(-j))}\in
U({\hat{{\mathfrak g}}_1})=S({\hat{{\mathfrak g}}_1})
$$
the corresponding monomials. We can identify $\pi$ with a sequence
$$
a_1, b_1, c_1, a_2, b_2, c_2,\dots
$$
with finitely many non-zero terms $a_j=\pi(x_{2}(-j))$,
$b_j=\pi(x_{0}(-j))$, $c_j=\pi(x_{\underline{2}}(-j))$ and
$$
x(\pi)=\ \dots\, x_{\underline{2}}(-j)\sp{c_j}x_{0}(-j)\sp{b_j}
x_{2}(-j)\sp{a_j}\dots\,x_{\underline{2}}(-1)\sp{c_1}x_{0}(-1)\sp{b_1}
x_{2}(-1)\sp{a_1}\,.
$$
For a monomial $x(\pi)$ we say that $x(\pi)v_\Lambda\in W(\Lambda)$
is a monomial vector. The main result of this paper is the
following:
\begin{theorem}\label{T:the main theorem} The set of monomial
vectors $x(\pi)v_\Lambda$ satisfying difference conditions
\begin{equation}\label{E:difference conditions}\begin{aligned}
c_{j+1}+b_{j+1}+c_{j}&\leq k,\\
b_{j+1}+a_{j+1}+c_{j}&\leq k,\\
a_{j+1}+c_{j}+b_{j}&\leq k,\\
a_{j+1}+b_{j}+a_{j}&\leq k
\end{aligned}
\end{equation}
for all $j\geq 1$, and initial conditions
\begin{equation}\label{E:initial conditions}
a_1\leq k_0, \quad b_1+a_{1}\leq k_0+k_2, \quad c_1+b_{1}\leq
k_0+k_2,
\end{equation}
 is a basis of level $k$ Feigin-Stoyanovsky type subspace $W(\Lambda)$.
\end{theorem}

\section{Difference conditions and initial conditions}

By Poincar\' e-Birkhoff-Witt theorem we have a spanning set of
monomial vectors $x(\pi)v_\Lambda$ in Feigin-Stoyanovsky type level
$k$ subspace $W(\Lambda)$. To reduce this spanning set to a basis
described in Theorem~\ref{T:the main theorem} we use vertex operator
algebra relations
$$
 x_\theta(z)\sp{k+1}=\sum_{n\in\mathbb Z}\left(\sum_{j_1+\dots+j_{k+1}=n}
 x_\theta(j_1)\dots x_\theta(j_{k+1})\right)z\sp{-n-k-1}=0\quad
 \text{on}\quad L(\Lambda)
$$
and its consequences $U(\mathfrak g_0)\cdot x_\theta(z)\sp{k+1}=0$,
where $\cdot$ denotes the adjoint action of $\mathfrak g_0$ on
$\hat{\mathfrak g}_1$. That is, the adjoint action of $\mathfrak
g_0$ on coefficients of formal Laurent series $x_\theta(z)\sp{k+1}$
gives relations
$$
\sum c_\mu x(\mu)=0\quad\text{on}\quad W(\Lambda),
$$
where for each relation the sum is over an infinite set of colored
partitions $\mu$. By choosing a proper order on the set of
monomials, for each sum we can determine the smallest term
$x(\rho)$, the so called leading term, which can be replaced on
$W(\Lambda)$ by a sum of higher (bigger) terms. The list of leading
terms is
\begin{equation}\label{E:leading terms}
\begin{aligned}
x_{\underline{2}}(-j-1)\sp{c_{j+1}}x_{0}(-j-1)\sp{b_{j+1}}
x_{\underline{2}}(-j)\sp{c_{j}},
\quad &c_{j+1}+b_{j+1}+c_{j} =k+1,\\
x_{0}(-j-1)\sp{b_{j+1}}x_{2}(-j-1)\sp{a_{j+1}}
x_{\underline{2}}(-j)\sp{c_{j}},
\quad &b_{j+1}+a_{j+1}+c_{j}=k+1,\\
x_{2}(-j-1)\sp{a_{j+1}}x_{\underline{2}}(-j)\sp{c_{j}}
x_{0}(-j)\sp{b_{j}},
\quad &a_{j+1}+c_{j}+b_{j}=k+1,\\
x_{2}(-j-1)\sp{a_{j+1}}x_{0}(-j)\sp{b_{j}} x_{2}(-j)\sp{a_{j}},
\quad &a_{j+1}+b_{j}+a_{j}=k+1
\end{aligned}
\end{equation}
for all $j\geq 1$. So by induction we see that $W(\Lambda)$ is
spanned by monomial vectors $x(\pi)v_\Lambda$ which don't have
factors of the form \eqref{E:leading terms}, i.e., by monomial
vectors which satisfy difference conditions \eqref{E:difference
conditions} (for details of this argument see \cite{LP}, \cite{MP2},
\cite{P2} or \cite{FKLMM}).

\begin{lemma}\label{L:initial conditions for W(Lambda1)}
\
$x_2(-1)v_{\Lambda_1}=x_0(-1)v_{\Lambda_1}=x_{\underline{2}}(-1)v_{\Lambda_1}=0$.
\end{lemma}
\begin{proof}
For $\alpha\in R$ denote by ${\mathfrak sl}_2(\alpha)\subset
\mathfrak g$ a Lie subalgebra generated with $x_\alpha$ and
$x_{-\alpha}$ and by
\begin{equation}\label{E:definition of sl2(alpha)}
\widehat{\mathfrak sl}_2(\alpha)=\sum_{n\in\mathbb Z}{\mathfrak
sl}_2(\alpha)\otimes t^{n}+\mathbb C c+\mathbb C d\subset
\hat{\mathfrak g}
\end{equation}
denote the corresponding affine Lie algebra of type $A_1\sp{(1)}$.
Note that for a level one $\hat{\mathfrak g}$-module $V$ the
restriction to $\widehat{\mathfrak sl}_2(\alpha)$ is a level one
representation if $\alpha $ is a long root, and it is a level two
representation if $\alpha $ is a short root. Also note that
$U(\widehat{\mathfrak sl}_2(\alpha))v_{\Lambda_1}$ is a standard
$A_1\sp{(1)}$-module and that its ${\mathfrak
sl}_2(\alpha)$-submodule on the top is a submodule of
$5$-dimensional vector representation for $B_2$.

In the case $\alpha=\varepsilon_1-\varepsilon_2={\underline{2}}$ we
have level one representation on $U(\widehat{\mathfrak
sl}_2(\alpha))v_{\Lambda_1}$ with $2$-dimensional ${\mathfrak
sl}_2(\alpha)$-module on the top, so it must be the standard
$A_1\sp{(1)}$-module $L(\Lambda_1)$. Hence
$x_{\varepsilon_1-\varepsilon_2}(-1)v_{\Lambda_1}=0$. Similarly
$x_\alpha(-1)v_{\Lambda_1}=0$ for
$\alpha=\varepsilon_1+\varepsilon_2$. On the other hand in the case
$\alpha=\varepsilon_1$ we have level two representation on
$U(\widehat{\mathfrak sl}_2(\alpha))v_{\Lambda_1}$ with
$3$-dimensional ${\mathfrak sl}_2(\alpha)$-module on the top, so it
must be the standard $A_1\sp{(1)}$-module $L(2\Lambda_1)$. Hence
again $x_\alpha(-1)v_{\Lambda_1}=0$.
\end{proof}

\begin{lemma}\label{L:initial conditions for W(Lambda0)} We have
\begin{enumerate}
\item $x_2(-1)x_{{2}}(-1)v_{\Lambda_0}=x_{\underline{2}}(-1)x_{\underline{2}}(-1)v_{\Lambda_0}=0,$
\item $x_0(-1)x_{{2}}(-1)v_{\Lambda_0}=x_{\underline{2}}(-1)x_0(-1)v_{\Lambda_0}=0,$
\item $x_{\underline{2}}(-1)x_2(-1)v_{\Lambda_0}=Cx_0(-1)x_0(-1)v_{\Lambda_0}$
for some $C\neq0$,
\item $x_0(-1)\sp3v_{\Lambda_0}=0.$
\end{enumerate}
\end{lemma}
\begin{proof}
Note that $x_\gamma(j)v_{\Lambda_0}=0$ for all $j\geq 0$, so the
relation $x_\gamma(z)\sp2=0$ on $L(\Lambda_0)$ for a long root
$\gamma$ implies
$$
x_\gamma(-1)\sp2v_{\Lambda_0}=\left(x_\gamma(-1)x_\gamma(-1)+2x_\gamma(-2)x_\gamma(0)+
2x_\gamma(-3)x_\gamma(1)+\dots\right)v_{\Lambda_0}=0
$$
and (1) follows. Since
$x_{\varepsilon_2}(0)v_{\Lambda_0}=x_{-\varepsilon_2}(0)v_{\Lambda_0}=0$,
the action of $x_{\varepsilon_2}(0)$ or $x_{-\varepsilon_2}(0)$ on
(1) gives relations (2) and (3). Since for level one $\hat{\mathfrak
g}$-module $V$ the restriction to $\widehat{\mathfrak sl}_2(\alpha)$
is level two representation if $\alpha $ is a short root, on
$L(\Lambda_0)$ we have a relation $x_0(z)\sp3=0$ and (4) follows.
\end{proof}
We fix vectors $w_2=v_{\Lambda_2}$ and $w_{\underline{2}}$ with
weights
$$
\omega_2=\tfrac12(\varepsilon_1+\varepsilon_2)\quad\text{and}\quad
\omega_{\underline{2}}=\tfrac12(\varepsilon_1-\varepsilon_2)
$$
in the $4$-dimensional spinor $\mathfrak g$-module on the top of
$L(\Lambda_{2})$. By using arguments as above we obtain the
following:
\begin{lemma}\label{L:initial conditions for W(Lambda2)}
We have
\begin{enumerate}
\item $x_{2}(-1)v_{\Lambda_2}=0$ \ and \
$x_{\underline{2}}(-1)w_{\underline{2}}=0$,
\item $x_{\underline{2}}(-1)x_{\underline{2}}(-1)v_{\Lambda_2}=
x_{\underline{2}}(-1)x_{0}(-1)v_{\Lambda_2}=x_{0}(-1)x_{0}(-1)v_{\Lambda_2}=0$.
\end{enumerate}
\end{lemma}
\begin{lemma}\label{L:spanning for W(Lambda)}
The set of monomial vectors $x(\pi)v_\Lambda$ satisfying difference
conditions \eqref{E:difference conditions} and initial conditions
\eqref{E:initial conditions} span $W(\Lambda)$.
\end{lemma}
\begin{proof} We have already mentioned how the relation
$x_\theta(z)\sp{k+1}=0$ on level $k$ standard module $L(\Lambda)$
leads to a spanning set of monomial vectors satisfying difference
conditions \eqref{E:difference conditions}. By following an idea
from \cite{T} we reduce the problem of initial conditions
\eqref{E:initial conditions} for level $k$ Feigin-Stoyanovsky type
subspace to a problem of difference conditions for level $k'<k$
Feigin-Stoyanovsky type subspace: we shall consider $\hat{\mathfrak
g}$-submodules in tensor products
$$
L(\Lambda_0)\sp{\otimes k_0} \otimes L(\Lambda_{2})\sp{\otimes k_2}
\quad\text{and}\quad L(\Lambda_1)\sp{\otimes k_1} \otimes\left(
L(\Lambda_{0})\sp{\otimes k_0}\otimes L(\Lambda_{2})\sp{\otimes
k_2}\right)
$$
of levels $k'=k_0+k_2$ and $k=k_0+k_1+k_2$\, generated by highest
weight vectors
$$
v_{\Lambda_0}\sp{\otimes k_0}\otimes v_{\Lambda_{2}}\sp{\otimes
k_2}\quad\text{and}\quad v_{\Lambda_1}\sp{\otimes k_1}\otimes\left(
v_{\Lambda_0}\sp{\otimes k_0}\otimes v_{\Lambda_{2}}\sp{\otimes
k_2}\right).
$$
Assume that for
$$
x(\pi)=x_{\underline{2}}(-1)\sp{c_1}x_{0}(-1)\sp{b_1}
x_{2}(-1)\sp{a_1}
$$
the monomial vector $x(\pi)v_\Lambda$ does not satisfy initial
conditions \eqref{E:initial conditions}
$$
a_1\leq k_0, \quad b_1+a_{1}\leq k_0+k_2, \quad c_1+b_{1}\leq
k_0+k_2.
$$
By the above lemmas
$$
x_2(-1)v_{\Lambda_1}=0,\quad x_2(-1)v_{\Lambda_2}=0,\quad
x_2(-1)\sp2v_{\Lambda_0}=0,
$$
so in the case when $a_1>k_0$ we have that the vector
$$
x_2(-1)\sp{a_1}\left(v_{\Lambda_1}\sp{\otimes k_1}\otimes
v_{\Lambda_0}\sp{\otimes k_0}\otimes v_{\Lambda_{2}}\sp{\otimes
k_2}\right)=x_2(-1)\sp{a_1-k_0}\left(v_{\Lambda_1}\sp{\otimes
k_1}\otimes \left(x_2(-1)v_{\Lambda_0}\right)\sp{\otimes k_0}\otimes
v_{\Lambda_{2}}\sp{\otimes k_2}\right)
$$
equals zero and we may omit it from our spanning set of monomial
vectors.

Now assume that the monomial vector $x(\pi)v_\Lambda$ does not
satisfy initial conditions because
$$
 k''=b_1+a_{1}> k_0+k_2.
$$
Then we have a relation
$$
x_2(z)\sp{b_1+a_{1}}=0\quad\text{on}\quad
L(k_0\Lambda_0+k_2\Lambda_2)\subset L(\Lambda_0)\sp{\otimes k_0}
\otimes L(\Lambda_{2})\sp{\otimes k_2},
$$
and by the adjoint action of $(x_{-\varepsilon_2})\sp{b_1}$ we get
$$
x_0(z)\sp{b_1}x_2(z)\sp{a_{1}}+\dots+
c_{u,v,t}\,x_{\underline{2}}(z)\sp{u}x_0(z)\sp{v}x_2(z)\sp t+\dots=0
$$
with $a_1<t$. The coefficient of $z\sp0$ gives us a relation
$$
R=x_0(-1)\sp{b_1}x_2(-1)\sp{a_{1}}+\dots+
c'_{u,v,t}\,x_{\underline{2}}(-1)\sp{u}x_0(-1)\sp{v}x_2(-1)\sp
t+\dots=0
$$
on $L(k_0\Lambda_0+k_2\Lambda_2)$. The coefficient $R$ is an
infinite sum with the leading term
\begin{equation}\label{E:leading term for initial condition}
x_0(-1)\sp{b_1}x_2(-1)\sp{a_{1}}.
\end{equation}
In $R$ we have monomials of the form $x_{\gamma_1}(j_1)\dots
x_{\gamma_{k''}}(j_{k''})$ with $j_1+\dots+j_{k''}=-k''$, so either
$j_1=\dots=j_{k''}=-1$ or we have $j_s\geq 0$ for some $s$. Hence
Lemma~\ref{L:initial conditions for W(Lambda1)} and
$$
x_\gamma(j)v_{\Lambda_i}=0\quad\text{for all}\quad \gamma\in\Gamma,\
j\geq 0 \quad\text{and}\quad  i=0,1,2
$$
imply
$$
Rv_{\Lambda}= R\left(v_{\Lambda_1}\sp{\otimes k_1}\otimes\left(
v_{\Lambda_0}\sp{\otimes k_0}\otimes v_{\Lambda_{2}}\sp{\otimes
k_2}\right)\right)= v_{\Lambda_1}\sp{\otimes k_1}\otimes R\left(
v_{\Lambda_0}\sp{\otimes k_0}\otimes v_{\Lambda_{2}}\sp{\otimes
k_2}\right)=0.
$$
Since the monomial \eqref{E:leading term for initial condition} is
the leading term of the relation $Rv_{\Lambda}=0$, we can express
$$
x_0(-1)\sp{b_1}x_2(-1)\sp{a_{1}}v_{\Lambda}
$$
as a combination of higher monomial vectors and we may omit it from
the spanning set. In a similar way we argue in the case when
$c_1+b_{1}> k_0+k_2$.
\end{proof}

\section{Simple current operators}

Recall that we have fixed a cominimal coweight
$\omega=\omega_1=\varepsilon_1\in\mathfrak h$. We shall use simple
current operators $[\omega]$ on level $1$ modules, i.e. linear
bijections
\begin{equation*}
L(\Lambda_0)\overset{[\omega]}\longrightarrow
L(\Lambda_{1})\overset{[\omega]}\longrightarrow
L(\Lambda_{0}),\qquad L(\Lambda_{2})
\overset{[\omega]}\longrightarrow L(\Lambda_{2})
\end{equation*}
such that
$$
x_\alpha (z)[\omega]=[\omega]z\sp{\alpha(\omega)}x_\alpha(z) \quad
\text{for all}\quad \alpha\in R
$$
or, written by components,
\begin{equation}\label{commutation of omega and x alpha}
x_\alpha (n)[\omega]=[\omega]x_\alpha(n+\alpha(\omega)) \quad
\text{for all}\quad \alpha\in R, \ n\in\mathbb Z.
\end{equation}

\begin{remark}
It is easy to see that, up to a scalar multiple, the linear
bijection $[\omega]$ between two irreducible modules is uniquely
determined by \eqref{commutation of omega and x alpha}. We can prove
the existence of such a map in several ways.

In the ADE case, when the lattice construction of level one $\hat
g$-modules is available, for minimal weight $\omega$ we have
$$
[\omega]\sim e\sp\omega
$$
(see \cite{G}, \cite{CLM1} or \cite{P3} for notation and details).
For level $k$  modules we consider a tensor product of $k$ level one
modules and we have
$$
[\omega]\sim e\sp\omega\otimes\dots\otimes e\sp\omega
$$
such that \eqref{commutation of omega and x alpha} holds.

Haisheng Li pointed out that in general the map $[\omega]$ can be
interpreted in terms of simple currents. In \cite{DLM} a module $M$
for a vertex operator algebra $V$ is called a simple current if the
tensor functor ``$M\boxtimes\ \cdot\ $'' is a bijection on the set
of equivalence classes $\text{\rm Irr} (V)$ of irreducible
$V$-modules. In \cite{DLM} simple currents $M$ for affine Lie
algebras are constructed by deforming vertex operators $Y_V(\cdot,
z)$ for simple vertex operator algebras $V=L(k\Lambda_0)$ with
formal Laurent series
$$
\Delta(\omega,z)=z\sp\omega  \exp -\biggl(\sum_{n>0} \omega (n)
(-z)^{- n}\big/ n\biggr)
$$
so that
$$
Y_{M}(\cdot,z)=Y_{V}(\Delta(\omega,z)\cdot,z).
$$

To prove the existence of $[\omega]$ for $B_2\sp{(1)}$ we may use a
related Dong-Li-Mason's result that for a representation
$L(\Lambda)$ of \,$\hat{\mathfrak g}$ on a vector space $W$,
realized by a vertex operator $Y_{L(\Lambda)}(\cdot,z)$, we also
have another representation $L(\Lambda')$ on the same vector space
$W$, but realized by a deformed vertex operator
$$
Y_{L(\Lambda')}(\cdot,z)=Y_{L(\Lambda)}(\Delta(\omega,z)\cdot,z)
$$
(cf. \cite{DLM} and \cite{L2}). Then
$$
[\omega]\colon L(\Lambda)\to L(\Lambda')
$$
can be interpreted as the identity map
$$
\text{id\,}\colon W\to W
$$
on the vector space $W$ endowed with two different structures of
$\hat{\mathfrak g}$-modules, $L(\Lambda)$ and $L(\Lambda')$.

We can also prove the existence of $[\omega]$ for $B_2\sp{(1)}$ by
following the approach of J.~Fuchs in \cite{Fu}: a representation
$L(\Lambda)$ of \,$\hat{\mathfrak g}$ on a vector space $W$, given
by
$$
\pi\colon\hat{\mathfrak g}\to \text{\rm End}\, (W),
$$
can be changed to a new representation $L(\Lambda')$ on the same
vector space $W$ by considering a composition
$$
\pi\circ\sigma\colon\hat{\mathfrak g}\to \text{\rm End}\, (W)
$$
of representation $\pi$ with an automorphism $\sigma$ of
$\hat{\mathfrak g}$ defined by
$$
\sigma\left(x_\alpha (n)\right)=x_\alpha(n+\alpha(\omega)) \quad
\text{for all}\quad \alpha\in R, \ n\in\mathbb Z.
$$
Then again $[\omega]\colon L(\Lambda)\to L(\Lambda')$ can be
interpreted as the identity map on $W$.
\end{remark}

\begin{remark} In our later arguments by induction on degree, we use the
map $[\omega]$ in essentially the same way as it is used in \cite{G}
and \cite{CLM1}: we ``move'' monomial vectors from one space to
another and, due to \eqref{commutation of omega and x alpha}, we
``lower'' their degrees in the process. For this reason we use the
same notation $[\omega]$ for all these different maps on different
spaces, including the corresponding maps on tensor products of level
one modules and on the symmetric algebra of level one modules (cf.
equation \eqref{commutation of omega and x} and Remark~\ref{R:omega
on symm.algebra} bellow). It should be noted that $[\omega]$
``behaves like a group element'' (cf. Remark~\ref{R:omega is a 4th
root of e} bellow).
\end{remark}

We fix $v_{\Lambda_0}=\bf 1$ in the vertex operator algebra
$L(\Lambda_{0})$. Then we have
\begin{lemma}\label{L:simple current initial conditions for 0 and 1} With properly
normalized $v_{\Lambda_1}$ and $x_2$
\begin{enumerate}
\item $[\omega]v_{\Lambda_0}=v_{\Lambda_1}$,
\item
$[\omega]v_{\Lambda_1}=x_{\underline{2}}(-1)x_2(-1)v_{\Lambda_0}$.
\end{enumerate}
\end{lemma}
\begin{proof}
(1) For a level $k$ standard $\hat{\mathfrak g}$-module $L(\Lambda)$
the new module structure
 $Y_{L(\Lambda')}(\cdot,z)$ $=Y_{L(\Lambda)}(\Delta(\omega,z)\cdot,z)$ gives
$$
h(0)[\omega]v_\Lambda=[\omega]\big(h(0)+\langle \omega, h\rangle
k\big)v_\Lambda\quad\text{for}\quad h\in\mathfrak h.
$$
In particular, $[\omega]v_{\Lambda_0}$ is a weight vector with level
$1$ weight $\Lambda_1=\Lambda_0+\langle\omega, \cdot\rangle$.
Relation \eqref{commutation of omega and x alpha} gives
$$
\begin{aligned}
&x_{-\theta}(1)[\omega]v_{\Lambda_0}=[\omega]x_{-\theta}(1-\theta(\omega))v_{\Lambda_0}
=[\omega]x_{-\theta}(0)v_{\Lambda_0}=0 \quad\text{and}\\
&x_{\alpha_i}(0)[\omega]v_{\Lambda_0}=[\omega]x_{\alpha_i}(0+\alpha_i(\omega))v_{\Lambda_0}
=[\omega]x_{\alpha_i}(\delta_{i1})v_{\Lambda_0}=0
\quad\text{for}\quad i=1,2.\\
\end{aligned}
$$
Hence $[\omega]v_{\Lambda_0}$ is a highest weight vector and
$L(\Lambda'_0)=L(\Lambda_1)$.

(2) Like in (1) we first see that
$[\omega]\sp{-1}x_{\underline{2}}(-1)x_2(-1)v_{\Lambda_0}$ is a
weight vector with weight $\Lambda_1$. By using \eqref{commutation
of omega and x alpha} and Lemma~\ref{L:initial conditions for
W(Lambda0)} we obtain
$$
\begin{aligned}
&x_{-\theta}(1)[\omega]\sp{-1}x_{\underline{2}}(-1)x_2(-1)v_{\Lambda_0}=
[\omega]\sp{-1}x_{-\theta}(2)x_{\underline{2}}(-1)x_2(-1)v_{\Lambda_0}
=0,\\
&x_{\alpha_1}(0)[\omega]\sp{-1}x_{\underline{2}}(-1)x_2(-1)v_{\Lambda_0}=
[\omega]\sp{-1}x_{\alpha_1}(-1)x_{\underline{2}}(-1)x_2(-1)v_{\Lambda_0}=0,
\\
&x_{\alpha_2}(0)[\omega]\sp{-1}x_{\underline{2}}(-1)x_2(-1)v_{\Lambda_0}=
[\omega]\sp{-1}x_{\alpha_2}(0)x_{\underline{2}}(-1)x_2(-1)v_{\Lambda_0}=0.
\end{aligned}
$$
Hence (2) holds and $L(\Lambda'_1)=L(\Lambda_0)$.
\end{proof}

\begin{lemma}\label{L:simple current initial conditions for 2} With properly
normalized $[\omega]$, $w_{\underline{2}}$ and $x_0,
x_{\underline{2}}$
\begin{enumerate}
\item
$[\omega]v_{\Lambda_2}=x_0(-1)v_{\Lambda_2}=x_2(-1)w_{\underline{2}}$,
\item
$[\omega]w_{\underline{2}}=x_0(-1)w_{\underline{2}}=x_{\underline{2}}(-1)v_{\Lambda_2}$.
\end{enumerate}
\end{lemma}
\begin{proof} (1) As in the proof of previous lemma we see that
$[\omega]\sp{-1}x_0(-1)v_{\Lambda_2}$ is a weight vector with weight
$\Lambda_2$. By using \eqref{commutation of omega and x alpha} and
Lemma~\ref{L:initial conditions for W(Lambda2)} we obtain
$$
\begin{aligned}
&x_{-\theta}(1)[\omega]\sp{-1}x_0(-1)v_{\Lambda_2}=
[\omega]\sp{-1}x_{-\theta}(2)x_0(-1)v_{\Lambda_0}
=0,\\
&x_{\alpha_1}(0)[\omega]\sp{-1}x_0(-1)v_{\Lambda_2}=
[\omega]\sp{-1}x_{\alpha_1}(-1)x_0(-1)v_{\Lambda_2}=0,
\\
&x_{\alpha_2}(0)[\omega]\sp{-1}x_0(-1)v_{\Lambda_2}=
[\omega]\sp{-1}x_{\alpha_2}(0)x_0(-1)v_{\Lambda_2}=0.
\end{aligned}
$$
Hence, with a proper normalization,
$[\omega]\sp{-1}x_0(-1)v_{\Lambda_2}=v_{\Lambda_2}$. The second
equality follows from Lemma~\ref{L:initial conditions for
W(Lambda2)} because
$$
0=x_{-\alpha_2}(0)0=x_{-\alpha_2}(0)x_2(-1)v_{\Lambda_2}
=C'x_0(-1)v_{\Lambda_2}+C''x_2(-1)w_{\underline{2}}
$$
for some $C',C''\neq 0$.

(2) The first equality follows from (1) by using the fact that
$w_{\underline{2}}$ is proportional to
$x_{-\alpha_2}(0)v_{\Lambda_2}$ and the fact that $x_{-\alpha_2}(0)$
commutes with $[\omega]$. The second equality follows from
Lemma~\ref{L:initial conditions for W(Lambda2)}.
\end{proof}

We define a linear bijection $[\omega]$ on a tensor product of $k$
fundamental modules as
$$
[\omega]\otimes\dots \otimes[\omega]\colon\bigotimes_{s=1}\sp k
L(\Lambda_{i_s})\to \bigotimes_{s=1}\sp k L(\Lambda_{i'_s}).
$$
It is clear that relation \eqref{commutation of omega and x alpha}
holds for $[\omega]=[\omega]\otimes\dots \otimes[\omega]$. In
particular,
\begin{equation}\label{commutation of omega and x}
x_\gamma (n)[\omega]=[\omega]x_\gamma(n+1) \quad \text{for}\quad
\gamma\in\Gamma.
\end{equation}
For a colored partition $\mu$ we set
$\mu\sp{+}(x_\gamma(n+1))=\mu(x_\gamma(n))$. Then for monomials
relation \eqref{commutation of omega and x} reads as
\begin{lemma}\label{L:commutation of omega and monomials}
\quad$x(\mu)[\omega]=[\omega]x(\mu\sp+)$.
\end{lemma}
\smallskip

\begin{remark}\label{R:notation for shifted partitions}
For $x(\mu)=\prod x_\gamma(n)\sp{m_\gamma(n)}$ we have
$x(\mu\sp+)=\prod x_\gamma(n+1)\sp{m_\gamma(n)}$, so we may say that
$x(\mu\sp+)$ is obtained from a monomial $x(\mu)$ by ``shifting
degrees of factors'' $x_\gamma(n)\to x_\gamma(n+1)$. Later on we
shall also use a notation
$\mu\sp{p}(x_\gamma(n+p))=\mu(x_\gamma(n))$ for any for $p\in\mathbb
Z$, and we shall write $\mu\sp{+p}$ when we want to emphasize the
shift of degrees of factors.
\end{remark}

From Lemmas \ref{L:simple current initial conditions for 0 and 1},
\ref{L:simple current initial conditions for 2}, \ref{L:initial
conditions for W(Lambda1)}, \ref {L:initial conditions for
W(Lambda0)} and \ref{L:initial conditions for W(Lambda2)} we have
\begin{equation}\label{E:v(Lambda to Lambda star)}
\begin{aligned}
&[\omega]\left(v_{\Lambda_0}\sp{\otimes k_0}\otimes
v_{\Lambda_1}\sp{\otimes k_1}\otimes v_{\Lambda_{2}}\sp{\otimes
k_2}\right)\\
&=\big([\omega]v_{\Lambda_0}\big)\sp{\otimes k_0}\otimes
\big([\omega]v_{\Lambda_1}\big)\sp{\otimes k_1}\otimes
\big([\omega]v_{\Lambda_2}\big)\sp{\otimes
k_2}\\
&=v_{\Lambda_1}\sp{\otimes k_0}\otimes
\big(x_{\underline{2}}(-1)x_2(-1)v_{\Lambda_0}\big)\sp{\otimes
k_1}\otimes
\big(x_0(-1)v_{\Lambda_{2}}\big)\sp{\otimes k_2}\\
&=C\,x_{\underline{2}}(-1)\sp{k_1}x_0(-1)\sp{k_2}x_2(-1)\sp{k_1}\left(v_{\Lambda_1}\sp{\otimes
k_0}\otimes v_{\Lambda_0}\sp{\otimes k_1}\otimes
v_{\Lambda_{2}}\sp{\otimes k_2}\right).
\end{aligned}
\end{equation}
For
\begin{equation}\label{E:Lambda to Lambda star}
\Lambda=k_0\Lambda_0+k_1\Lambda_1+k_2\Lambda_2\quad\text{set}\quad
\Lambda\sp*=k_1\Lambda_0+k_0\Lambda_1+k_2\Lambda_2.
\end{equation}
Then \eqref{E:v(Lambda to Lambda star)} and Lemma~\ref{L:commutation
of omega and monomials} imply
\begin{proposition}\label{P: omega preserves W's}
\quad $[\omega]\colon L(\Lambda)\to
L(\Lambda\sp*)$\quad\text{and}\quad $[\omega]\colon W(\Lambda)\to
W(\Lambda\sp*)$.
\end{proposition}
\noindent This proposition and a construction in \cite{DLM} and
\cite{L2} show that $$[\omega]=[\omega]\otimes\dots
\otimes[\omega]$$ is a simple current operator for level $k$
standard modules.
\medskip

Virasoro algebra operators in a vertex operator algebra are usually
denoted by $L(n)$, $n\in \mathbb Z$. If we set
$L(0)v_\Lambda=C_\Lambda v_\Lambda$, then
$$
d=-L(0)+C_\Lambda\quad\text{on}\quad L(\Lambda).
$$
We have the following:
\begin{lemma}\label{L: degree for level k omega} For elements  $h$ of the Cartan
subalgebra $\mathfrak h$, and the Virasoro algebra element $L(0)$,
on level $k$ standard modules we have
\begin{enumerate}
\item\quad $[\omega]\sp{-n}h(0)\,[\omega]\sp{n}=h(0)+n\langle\omega, h\rangle
k$ \quad for all \quad  $n\in\mathbb Z$, and
\item\quad $[\omega]\sp{-n}L(0)\,[\omega]\sp{n}=L(0)+n\,\omega(0)+\tfrac{n\sp2}{2}\langle\omega, \omega\rangle
k$ \quad for all \quad  $n\in\mathbb Z$.
\end{enumerate}
\end{lemma}
\begin{proof} As it was already said, we can view $[\omega]$ as the
identity map on $L(\Lambda)\to L(\Lambda)$, where the target space
is given a new module structure $L(\Lambda')$ with a vertex operator
\begin{equation}\label{E:deformed vertex operators}
Y_{L(\Lambda')}(\cdot,z)=Y_{L(\Lambda)}(\Delta(\omega,z)\cdot,z),\
\Delta(\omega,z)=z\sp\omega  \exp -\biggl(\sum_{n>0} \omega (n)
(-z)^{- n}\big/ n\biggr).
\end{equation}
Then $L(0)[\omega]$ is the coefficient of $z\sp{-2}$ in the vertex
operator
$$
Y_{L(\Lambda')}(L(-2)\bold
1,z)=Y_{L(\Lambda)}\big(\Delta(\omega,z)\,L(-2)\bold 1,z\big),
$$
and only three terms in
$$
\Delta(\omega,z)=1-\left(\omega (1) (-z)^{- 1}+\omega (2) (-z)^{-
2}\big/ 2+\dots \right)+\tfrac{1}{2!}\left(\omega (1) (-z)^{-
1}+\dots\right)\sp2+\dots
$$
give a contribution to this coefficient
$$
L(0)[\omega]=[\omega]\,\big(L(0)+\omega(0)+\tfrac{1}{2}\langle\omega,
\omega\rangle k\big).
$$
Now (2) follows by induction. Relation (1) is proved in a similar
way.
\end{proof}
\begin{remark}
In the proofs of Lemmas~\ref{L:simple current initial conditions for
0 and 1}, \ref{L:simple current initial conditions for 2} and
\ref{L: degree for level k omega}(1) we suggested the use of
deformed vertex operators \eqref{E:deformed vertex operators}, but
all these statements can be proved by using formula
\eqref{commutation of omega and x alpha}) as well.

On the other side, the formula in Lemma~\ref{L: degree for level k
omega}(2) written for operator $d$,
$$
-d\,[\omega]\sp{n}=[\omega]\sp{n}
\left(-d+C_{\Lambda}-C_{\Lambda'}+n\,\omega(0)+\tfrac{n\sp2}{2}\langle\omega,
\omega\rangle k\right)\quad\text{if}\quad [\omega]\sp{n}v_\Lambda\in
L(\Lambda'),
$$
contains a term $C_{\Lambda}-C_{\Lambda'}$ which in general depends
on $L(0)$. When $\Lambda'=\Lambda$ the power $[\omega]\sp{n}$ is a
Weyl group translation operator on $L(\Lambda)$ (cf.
Lemma~\ref{L:Weyl translation and simple current operator} and
Remark~\ref{R:omega is a 4th root of e}) and a formula for $d$
follows from \eqref{E:Weyl translations normalize root vectors}.
\end{remark}

\section{Coefficients of level $1$ intertwining operators}

Let $V$ be a vertex operator algebra and let $W_1$, $W_2$ and $W_3$
be three $V$-modules. Then an intertwining operator $\mathcal Y$ of
type $\binom{W_3}{W_1\,W_2}$ is a formal series
$$
\mathcal Y(w,z)=\sum_{n\in \mathbb Q}w_nz\sp{-n-1}, \quad w\in W_1,
$$
with coefficients
$$
 w_n\in \text{Hom\,}(W_2,W_3)\quad\text{for}\quad n\in\mathbb Q
$$
such that ``all the defining properties of a module action that make
sense hold'' (see \cite{FHL}). In particular, for $v\in V$ we have a
commutator formula
$$
v_jw_n-w_nv_j=\sum_{i\geq 0}\binom{j}{i}(v_iw)_{n+j-i},
$$
where $v_j$ in $v_jw_n$ is a coefficient of the vertex operator
$Y_{W_3}(v,z)=\sum v_jz\sp{-j-1}$ for $V$-module $W_3$, $v_j$ in
$w_nv_j$ is a coefficient of the vertex operator $Y_{W_2}(v,z)=\sum
v_jz\sp{-j-1}$ for $V$-module $W_2$ and $v_i$ in $v_iw$ is a
coefficient of the vertex operator $Y_{W_1}(v,z)=\sum v_iz\sp{-i-1}$
for $V$-module $W_1$. The vector space of all intertwining operators
of type $\binom{W_3}{W_1\,W_2}$ is denoted by
$I\binom{W_3}{W_1\,W_2}$ and its dimension is called a fusion rule.
We have
\begin{equation*}
 I\binom{W_3}{W_1\,W_2}\cong I\binom{W_3}{W_2\,W_1}\cong
I\binom{W_2'}{W_1\,W_3'},
\end{equation*}
where for $V$-module $M$ we denote by $M'$ the contragredient module
(see \cite{FHL}). If $W_1$ is an irreducible $V$-module, $W_2$ a
simple current module and
$$
W_3=W_1\boxtimes W_2
$$
then by Lemma~2.3 in \cite{L1} the fusion
$$
\dim I\binom{W_3}{W_1\,W_2}=1.
$$

\begin{lemma}\label{L:coefficients of intertwining operators}
Let $\hat{\mathfrak g}$ be an affine Lie algebra and $L(k\Lambda_0)$
a vacuum level $k$ standard $\hat{\mathfrak g}$-module. Let $V_1$,
$V_2$ and $V_3$ be irreducible modules for vertex operator algebra
$V=L(k\Lambda_0)$. Let $\mathcal Y\neq 0$ be an intertwining
operator of type $\binom{V_3}{V_1\,V_2}$, let $W$ be the top of
$V_1$ and $v\neq 0$ a vector on the top of $V_2$. Then there is
$m\in \mathbb Q$ such that the top of $V_3$ is a $\mathfrak
g$-module
$$
U(\mathfrak g)\{w_m v\mid w\in W\}.
$$
\end{lemma}
\begin{proof} By Proposition~11.9 in \cite{DL} we have $\mathcal
Y(w,z)v\neq0$ for $w\neq0$ and, from the definition of intertwining
operators, $w_n v=0$ for all $n$ large enough. Let
$$
m=\max\,\{n\in\mathbb Q\mid w_nv\neq 0\ \text{for some}\  w\in W\}.
$$
Then we have a nonzero subspace
$$
\{w_m v\mid w\in W\}\subset V_3.
$$
For $x_j=x(j)$ in $\hat{\mathfrak g}$ we have a commutator formula
$$
x_jw_m-w_mx_j=\sum_{i\geq 0}\binom{j}{i}(x_iw)_{m+j-i}
$$
which for $j>0$ implies
$$
x_j(w_mv)=w_mx_jv+\sum_{i\geq
0}\binom{j}{i}(x_iw)_{m+j-i}v=(x_0w)_{m+j}v=0
$$
because $v$ and $w$ are vectors on the top of modules and $m$ is
maximal such that $w_n v$ can be nonzero. Since
$$
U(\hat{\mathfrak g}_{\leq0})\{w_m v\mid w\in W\}\subset V_3
$$
is an $\hat{\mathfrak g}$-invariant subspace of irreducible
$\hat{\mathfrak g}$-module $V_3$, the space $\{w_m v\mid w\in W\}$
must be a subspace of the top of $V_3$ and the lemma follows.
\end{proof}

Recall that we have fixed vectors $w_2=v_{\Lambda_2}$ and
$w_{\underline{2}}$ with weights $\omega_2$ and
$\omega_{\underline{2}}$ in the $4$-dimensional spinor $\mathfrak
g$-module on the top of $L(\Lambda_{2})$.

\begin{proposition}\label{P:coefficients of intertwining operators}
(1) With proper scalars $\lambda$ and $\mu$ and an intertwining
operator $\mathcal Y$ of type
\begin{equation*}
\binom{L(\Lambda_2)}{L(\Lambda_2)\quad L(\Lambda_0)}
\end{equation*}
there are coefficients \medskip

$[\omega_{2}]$\ of\quad $\mathcal Y(\lambda w_2,z)=\sum_{n\in
\mathbb Q}(\lambda w_2)_nz\sp{-n-1}$,\qquad $[\omega_{2}]\colon
L(\Lambda_0)\to L(\Lambda_{2})$,

$[\omega_{\underline{2}}]$\ of\quad $\mathcal Y(\mu
w_{\underline{2}},z)=\sum_{n\in \mathbb Q}(\mu
w_{\underline{2}})_nz\sp{-n-1}$,\qquad
$[\omega_{\underline{2}}]\colon L(\Lambda_0)\to L(\Lambda_{2})$,
\medskip

\noindent which commute with the action of $\hat{\mathfrak g}_1$ and
such that
$$
[\omega_{2}]v_{\Lambda_0}=v_{\Lambda_2},\qquad
[\omega_{\underline{2}}]v_{\Lambda_0}=w_{\underline{2}}.
$$
\medskip

(2) With proper scalars $\lambda$ and $\mu$ and an intertwining
operator $\mathcal Y$ of type
\begin{equation*}
\binom{L(\Lambda_1)}{L(\Lambda_2)\quad L(\Lambda_2)}
\end{equation*}
there are coefficients \medskip

$[\omega_{2}]$\ of\quad $\mathcal Y(\lambda w_2,z)=\sum_{n\in
\mathbb Q}(\lambda w_2)_nz\sp{-n-1}$,\qquad $[\omega_{2}]\colon
L(\Lambda_2)\to L(\Lambda_{1})$,

$[\omega_{\underline{2}}]$\ of\quad $\mathcal Y(\mu
w_{\underline{2}},z)=\sum_{n\in \mathbb Q}(\mu
w_{\underline{2}})_nz\sp{-n-1}$,\qquad
$[\omega_{\underline{2}}]\colon L(\Lambda_2)\to L(\Lambda_{1})$,
\medskip

\noindent which commute with the action of $\hat{\mathfrak g}_1$ and
such that
$$
[\omega_{2}]v_{\Lambda_{2}}=0,\quad
[\omega_{2}]w_{\underline{2}}=v_{\Lambda_{1}},\qquad
[\omega_{\underline{2}}]v_{\Lambda_2}=v_{\Lambda_1},\quad
[\omega_{\underline{2}}]w_{\underline{2}}=0.
$$
\end{proposition}

\begin{proof} Since $L(\Lambda_2)$ is $L(\Lambda_0)$-module we have
$$
I\binom{L(\Lambda_2)}{L(\Lambda_2)\quad L(\Lambda_0)}\cong
I\binom{L(\Lambda_2)}{L(\Lambda_0)\quad L(\Lambda_2)}
$$
and the space of intertwining operators of this type is
$1$-dimensional. Since $L(\Lambda_1)$ is a simple current module
such that
$$
L(\Lambda_1)\boxtimes L(\Lambda_2)=L(\Lambda_2)
$$ (see
\cite{L1}, \cite{L2} or \cite{DLM}), and since both $L(\Lambda_1)$
and $L(\Lambda_2)$ are self-dual, we have
$$
I\binom{L(\Lambda_1)}{L(\Lambda_2)\quad L(\Lambda_2)}\cong
I\binom{L(\Lambda_2)}{L(\Lambda_2)\quad L(\Lambda_1)}
$$
and the space of intertwining operators of this type is
$1$-dimensional.

Let $\mathcal Y\neq0$ be an intertwining operator of type
$\binom{L(\Lambda_2)}{L(\Lambda_2)\quad L(\Lambda_0)}$ and
$v=v_{\Lambda_0}$ on the top of $L(\Lambda_0)$. By
Lemma~\ref{L:coefficients of intertwining operators} there is a
vector $w$ on the top of $L(\Lambda_2)$ and an integer $m$ such that
$w_mv$ is proportional to $v_{\Lambda_2}$. It is clear that $w$ is
proportional to $v_{\Lambda_2}$ and we denote by $[\omega_2]=w_m$
the corresponding coefficient of the formal series $\mathcal
Y(v_{\Lambda_2},z)$. Obviously for proper normalization of $w$ we
have
$$
[\omega_2]v_{\Lambda_0}=v_{\Lambda_2}.
$$
On the other hand, if we take $w=w_{\underline{2}}$ and the
corresponding coefficient $[\omega_{\underline{2}}]=w_m$ of the
formal series $\mathcal Y(w_{\underline{2}},z)$, with proper
normalization we have
$$
[\omega_{\underline{2}}]v_{\Lambda_0}=w_{\underline{2}}.
$$

Now let $\mathcal Y\neq0$ be an intertwining operator of type
$\binom{L(\Lambda_1)}{L(\Lambda_2)\quad L(\Lambda_2)}$ and
$v=v_{\Lambda_2}$ on the top of $L(\Lambda_2)$. By
Lemma~\ref{L:coefficients of intertwining operators} there is a
vector $w$ on the top of $L(\Lambda_2)$ and an integer $m$ such that
vector $w_mv$ generates the irreducible $5$-dimensional $\mathfrak
g$-module on the top of $L(\Lambda_1)$. Since the top of
$L(\Lambda_2)$ is $4$-dimensional spinor $\mathfrak g$-module,
$\mathfrak h$-weight vectors of the form $w_mv$ can have weights
$$
\tfrac12(\varepsilon_1-\varepsilon_2)+\tfrac12(\varepsilon_1+\varepsilon_2),\quad
-\tfrac12(\varepsilon_1-\varepsilon_2)+\tfrac12(\varepsilon_1+\varepsilon_2),\quad
-\tfrac12(\varepsilon_1+\varepsilon_2)+\tfrac12(\varepsilon_1+\varepsilon_2).
$$
In the first case $w$ is proportional to $w_{\underline{2}}$ and
$w_mv=Cv_{\Lambda_1}$ for some scalar $C\neq0$. Vectors in the
second and third case can be transformed to the vector $w'_mv=C
v_{\Lambda_1}$ in the first case by acting with Lie algebra
$\mathfrak g$ elements $x_{\varepsilon_1-\varepsilon_2}$ and
$x_{\varepsilon_1}$ respectively. So if we take
$w=w_{\underline{2}}$ and the corresponding coefficient
$[\omega_{\underline{2}}]=w_m$ of the formal series $\mathcal
Y(w_{\underline{2}},z)$ with proper normalization, we have
$$
[\omega_{\underline{2}}]v_{\Lambda_2}=v_{1}.
$$
Inspection of $\mathfrak h$-weights in $5$-dimensional $\mathfrak
g$-module on the top of $L(\Lambda_1)$ shows that
$[\omega_{\underline{2}}]w_{\underline{2}}=0$. In a similar way we
see that for $w=v_{\Lambda_2}$ and the properly normalized
corresponding coefficient $[\omega_{2}]=w_m$ of the formal series
$\mathcal Y(v_{\Lambda_2},z)$ we have
$$
[\omega_{2}]w_{\underline{2}}=v_{\Lambda_1}\quad\text{and}\quad
[\omega_{2}]v_{\Lambda_2}=0.
$$

In each of the above cases $[\omega_{2}]$ and
$[\omega_{\underline{2}}]$ are coefficients of $\mathcal Y(w,z)$
with $w$ such that
$$
x(i)w=0\qquad\text{for all}\quad x\in\mathfrak g_1,\ i\geq 0.
$$
Hence the commutation relations for intertwining operators imply
$$
x(j)w_m-w_mx(j)=\sum_{i\geq
0}\binom{j}{i}(x(i)w)_{m+j-i}=0\quad\text{for all}\quad
x(j)\in\hat{\mathfrak g}_1.
$$
\end{proof}

\begin{remark}
In the Introduction we gave a very rough idea how coefficients of
intertwining operators are used in the proof of linear independence
of the monomial basis given by Theorem~\ref{T:the main theorem}:
with this operators we ``move'' monomial vectors
$x(\pi)v_{\Lambda}\to x(\pi)v_{\Lambda'}$ from one space to another
until we get vectors of the form $x(\pi')[\omega]v_{\Lambda''}$.
Since these operators commute with all $x(\pi)$, the only thing that
matters is how these operators ``move'' the highest weight vectors
$v_{\Lambda}\to v_{\Lambda'}$, in our case it is
$$
\begin{aligned}
&v_{\Lambda_0}\overset{[\omega_2]}\longrightarrow
v_{\Lambda_{2}}\overset{[\omega_{\underline{2}}]}\longrightarrow
v_{\Lambda_{1}},\quad [\omega_{\underline{2}}]w_{\underline{2}}=0,\\
&v_{\Lambda_0}\overset{[\omega_{\underline{2}}]}\longrightarrow
w_{\underline{2}}\overset{[\omega_2]}\longrightarrow
v_{\Lambda_{1}},\quad [\omega_{2}]v_{\Lambda_{2}}=0.
\end{aligned}
$$
For this reason it is convenient to use for different operators the
same symbol $[\omega_{2}]$ which reminds us only that they are
obtained as some coefficients of different series $\mathcal
Y(w_2,z)$ associated with the ``same'' vector $w_2$, or, to be
precise, associated with the same weight subspace of weight
$\omega_2$ of the top of $L(\Lambda_2)$.
\end{remark}

\section{Proof of linear independence}

By Lemma~\ref{L:spanning for W(Lambda)} the set of monomial vectors
$$
x(\pi)v_\Lambda=\ \dots\,
x_{\underline{2}}(-j)\sp{c_j}x_{0}(-j)\sp{b_j}
x_{2}(-j)\sp{a_j}\dots\,x_{\underline{2}}(-1)\sp{c_1}x_{0}(-1)\sp{b_1}
x_{2}(-1)\sp{a_1}v_\Lambda
$$
satisfying difference conditions \eqref{E:difference conditions} and
initial conditions \eqref{E:initial conditions} spans $W(\Lambda)$.
We prove linear independence of this set by induction on degree
$$
-n=\sum_{\gamma\in\Gamma, \, j\geq
1}-j\cdot\pi(x_\gamma(-j))=-\left(1a_1+1b_1+1c_1+\dots+ja_j+jb_j+jc_j+\dots\right)
$$
of monomials $x(\pi)$, considering in a proof all level $k$ modules
simultaneously. In the proof we shall briefly write DC for
difference conditions \eqref{E:difference conditions} and IC for
initial conditions \eqref{E:initial conditions}.
\bigskip

\noindent {\bf Step 1.}\ The idea of proof is illustrated most
clearly in a proof of linear independence of vectors
$$
x(\pi)v_{k\Lambda_1}
$$
of degree $-n$. As induction hypothesis we assume that vectors
$x(\mu)v_{k\Lambda_0}$ of degree $>-n$ are linearly independent.
Assume that
\begin{equation}\label{E:induction step 1}
\sum c_\pi x(\pi)v_{k\Lambda_1}=0.
\end{equation}
By Lemma~\ref{L:simple current initial conditions for 0 and 1} we
have $v_{\Lambda_1}=[\omega]v_{\Lambda_0}$ and hence
$v_{k\Lambda_1}=[\omega]v_{k\Lambda_0}$. By Lemma~\ref{L:commutation
of omega and monomials}
$$
\sum c_\pi x(\pi)v_{k\Lambda_1}=\sum c_\pi
x(\pi)[\omega]v_{k\Lambda_0}=[\omega]\sum c_\pi
x(\pi\sp+)v_{k\Lambda_0}
$$
and injectivity of $[\omega]$ implies
\begin{equation}\label{E:induction step 2}
\sum c_\pi x(\pi\sp+)v_{k\Lambda_0}=0.
\end{equation}
Monomials $x(\pi)$ in \eqref{E:induction step 1} satisfy difference
conditions, so obviously ``shifted by degree'' monomials
$x(\pi\sp+)$ in \eqref{E:induction step 2} satisfy difference
conditions as well. Monomials $x(\pi)$ in \eqref{E:induction step 1}
satisfy initial conditions for $k\Lambda_1$, i.e., contain no part
of the form $x_\alpha(-1)$. But then monomials $x(\pi\sp+)$ in
\eqref{E:induction step 2} contain parts of the form $x_\alpha(-j)$,
$j\geq 1$, and hence satisfy initial conditions for $k\Lambda_0$.
Since monomials $x(\pi\sp+)$ in \eqref{E:induction step 2} have
degrees $>-n$, the induction hypothesis implies that all $c_\pi=0$.
Hence we proved linear independence of monomial basis vectors for
$W(k\Lambda_1)$ of degree $-n$.
\bigskip

\noindent {\bf Step 2.}\ For ${\mathcal A}=(c_1,b_1,a_1)$ write
$$
x(-1)\sp{\mathcal A}=x_{\underline{2}}(-1)\sp{c_1}x_{0}(-1)\sp{b_1}
x_{2}(-1)\sp{a_1}.
$$
Later on it will be convenient to write a monomial $x(\mu)$ as a
product
$$
 \dots\, x_{\underline{2}}(-j)\sp{c_j}x_{0}(-j)\sp{b_j}
x_{2}(-j)\sp{a_j}\dots\,x_{\underline{2}}(-1)\sp{c_1}x_{0}(-1)\sp{b_1}
x_{2}(-1)\sp{a_1}=x(\mu_2)x(-1)\sp{\mathcal A_\mu}.
$$
We define a partial order on the set of level $k$ integral dominant
weights:
$$
\Lambda'=k'_0\Lambda_0+k'_1\Lambda_1+k'_2\Lambda_2\leq
\Lambda=k_0\Lambda_0+k_1\Lambda_1+k_2\Lambda_2
$$
if and only if
$$
\begin{aligned}
k'_0&\leq k_0, \\ k'_0+k'_2&\leq k_0+k_2.
\end{aligned}
$$
Clearly $k\Lambda_1$ is the smallest element and $k\Lambda_0$ is the
largest element in the set of level $k$ integral dominant weights.

Now we proceed with a proof of linear independence. We assume that
vectors $x(\mu)v_{\Lambda'}$ of degree $\geq-n$ satisfying DC and IC
are linearly independent for some set of $\Lambda'\geq k\Lambda_1$.
Let $\Lambda$ be a minimal level $k$ integral weight for which we
need to prove linear independence of monomial vectors of degree
$\geq-n$ satisfying DC and IC. Let
\begin{equation}\label{E:induction step 2'}
\sum c_\pi x(\pi)v_{\Lambda}=0.
\end{equation}
Assume that $c_\mu\neq 0$ for some $x(\mu)=x(\mu_2)x(-1)\sp{\mathcal
A_\mu}$ for
$$
\mathcal A_\mu=(c_1,b_1,a_1),\qquad a_1<k_0,
$$
and we assume that $a_1$ is the smallest power of $x_2(-1)$
appearing in such $A_\mu$. Since $[\omega_2]\colon L(\Lambda_0)\to
L(\Lambda_2)$, we have the operator
$$
1\sp{\otimes a_1}\otimes\,[\omega_2]\sp{\otimes(k_0-a_1)}\otimes
1\sp{\otimes(k_2+k_1)}\colon L(\Lambda)\to L(\Lambda'),
$$
$$
v_{\Lambda_0}\sp{\otimes k_0}\otimes v_{\Lambda_2}\sp{\otimes
k_2}\otimes v_{\Lambda_{1}}\sp{\otimes k_1} \mapsto
v_{\Lambda_0}\sp{\otimes a_1}\otimes v_{\Lambda_2}\sp{\otimes
(k_2+k_0-a_1)}\otimes v_{\Lambda_{1}}\sp{\otimes k_1}
$$
\smallskip

\noindent which commutes with the action of $\hat{\mathfrak g}_1$.
Note that $\Lambda> \Lambda'$ so that we may use the induction
hypothesis for corresponding monomial vectors. If we apply this
operator on the sum \eqref{E:induction step 2'} we get
\begin{equation}\label{E:induction step 2'1}
\sum c_\pi x(\pi)v_{\Lambda'}=0.
\end{equation}
By Lemmas~\ref{L:initial conditions for W(Lambda1)}, \ref{L:initial
conditions for W(Lambda2)}, \ref{L:initial conditions for
W(Lambda0)} we have $x_{2}(-1)v_{\Lambda_2}=0$,
$x_{2}(-1)v_{\Lambda_1}=0$, $x_2(-1)\sp2v_{\Lambda_0}=0$, so for any
monomial $x(\pi)= x(\pi') x_{2}(-1)\sp a$ with $a>a_1$ we have
$$
x(\pi)v_{\Lambda'}= x(\pi') x_{2}(-1)\sp a
\left(v_{\Lambda_0}\sp{\otimes a_1}\otimes v_{\Lambda_2}\sp{\otimes
(k_2+k_0-a_1)}\otimes v_{\Lambda_{1}}\sp{\otimes k_1}\right)=0.
$$
On the other hand, vectors like $x(\mu)v_{\Lambda'}$ besides DC
satisfy IC as well, i.e.,
$$
a_1\leq k'_0=a_1,\quad b_1+a_1\leq k'_0+k'_2=k_0+k_2,\quad
c_1+b_1\leq k'_0+k'_2=k_0+k_2,
$$
so by induction hypothesis the coefficient $c_\mu$ in linear
combination \eqref{E:induction step 2'1} must be zero --- a
contradiction.

So in \eqref{E:induction step 2'} we need to consider only monomials
with $a_1=k_0$, i.e., monomials of the form
$$
x(\pi)=  x(\pi') x_2(-1)\sp{k_0}.
$$
Assume that $c_\mu\neq 0$ for some $x(\mu)=x(\mu_2)x(-1)\sp{\mathcal
A_\mu}$ for
$$
\mathcal A_\mu=(c_1,b_1,a_1),\qquad a_1=k_0, \quad b_1+a_1<k_0+k_2,
\quad c_1+b_1<k_0+k_2.
$$
Since $[\omega_2]\colon L(\Lambda_2)\to L(\Lambda_1)$, we have the
operator
$$
1\sp{\otimes(k_0+k_2-1)}\otimes\,[\omega_2]\otimes 1\sp{\otimes
k_1}\colon L(\Lambda)\to L(\Lambda'),
$$
$$
v_{\Lambda_0}\sp{\otimes k_0}\otimes v_{\Lambda_2}\sp{\otimes
k_2}\otimes v_{\Lambda_{1}}\sp{\otimes k_1} \mapsto
v_{\Lambda_0}\sp{\otimes k_0}\otimes v_{\Lambda_2}\sp{\otimes
(k_2-1)}\otimes v_{\Lambda_{1}}\sp{\otimes (k_1+1)}
$$
\smallskip

\noindent which commutes with the action of $\hat{\mathfrak g}_1$.
Note that $\Lambda> \Lambda'$ so that we may use the induction
hypothesis for corresponding monomial vectors. If we apply this
operator on the sum \eqref{E:induction step 2'} we get
\begin{equation}\label{E:induction step 2''}
\sum c_\pi x(\pi)v_{\Lambda'}=0.
\end{equation}
By Lemmas~\ref{L:initial conditions for W(Lambda1)}, \ref{L:initial
conditions for W(Lambda2)}, \ref{L:initial conditions for
W(Lambda0)} we have $x_{2}(-1)v_{\Lambda_2}=0$,
$x_{2}(-1)v_{\Lambda_1}=0$, $x_2(-1)\sp2v_{\Lambda_0}=0$, so for any
monomial $x(\pi)= x(\pi') x_{2}(-1)\sp{k_0}$ we have
$$
\begin{aligned}
x(\pi)v_{\Lambda'}&= x(\pi') x_{2}(-1)\sp{k_0}
\left(v_{\Lambda_0}\sp{\otimes k_0}\otimes v_{\Lambda_2}\sp{\otimes
(k_2-1)}\otimes v_{\Lambda_{1}}\sp{\otimes (k_1+1)}\right) \\
&= C\,x(\pi') \left((x_{2}(-1)v_{\Lambda_0})\sp{\otimes k_0}\otimes
v_{\Lambda_2}\sp{\otimes (k_2-1)}\otimes v_{\Lambda_{1}}\sp{\otimes
(k_1+1)}\right)
\end{aligned}
$$
for some $C\neq0$. If for such $x(\pi)=x(\pi_2)x(-1)\sp{\mathcal
A_\pi}$ we have
$$
 b_1+a_1=k_0+k_2 \quad\text{or}\quad c_1+b_1=k_0+k_2,
$$
then by Lemma~\ref{L:initial conditions for W(Lambda2)} we have
$x_{\underline{2}}(-1)\sp2 v_{\Lambda_2}=
x_{\underline{2}}(-1)x_{0}(-1) v_{\Lambda_2}=x_{0}(-1)\sp2
v_{\Lambda_2}=0$ and by Lemma~\ref{L:initial conditions for
W(Lambda0)} we have $x_{\underline{2}}(-1)x_{0}(-1)
v_{\Lambda_0}=0$, so in either case at least one of $x_{0}(-1)$ or
$x_{\underline{2}}(-1)$ must act on one copy of $v_{\Lambda_1}$.
Hence by Lemma~\ref{L:initial conditions for W(Lambda1)} for such
$x(\pi)$ it must be
$$
x(\pi)v_{\Lambda'}=0.
$$
So in \eqref{E:induction step 2''} we have only vectors like
$x(\mu)v_{\Lambda'}$ which besides DC satisfy IC as well, and by
induction hypothesis the coefficient $c_\mu$ in linear combination
\eqref{E:induction step 2''} must be zero --- a contradiction.
\smallskip

\begin{remark}\label{R:omega on symm.algebra} For the rest of the proof
it will be convenient to realize
$L(\Lambda)$ of level $k$ in a $k\sp\text{th}$ component of a
symmetric algebra
$$
L(\Lambda)\subset S\sp k(V),\qquad V=L(\Lambda_0)\oplus
L(\Lambda_1)\oplus L(\Lambda_2).
$$
Operator $[\omega]=S([\omega])$ acts as ``a group element'' on
$S(V)$. On the other hand, operators $A$ and $B$ on $V$ in
Lemmas~\ref{L:proof of independence step 3A} and \ref{L:proof of
independence step 3B} below act as derivations on $S(V)$.
\end{remark}

\noindent {\bf Step 3.}\ By the previous step in the linear
combination \eqref{E:induction step 2'} we need to consider only
monomials $x(\pi)=x(\pi_2)x(-1)\sp{\mathcal A_\pi}$ with $a_1=k_0$
for  $\mathcal A_\pi=(c_1,b_1,a_1)$ and
$$
 b_1+a_1=k_0+k_2 \quad\text{or}\quad c_1+b_1=k_0+k_2.
$$
Assume first we have a monomial vector $x(\pi)v_\Lambda$ such that
$a_1=k_0$ and $b_1+a_1=k_0+k_2$ and  $c_1+b_1\leq k_0+k_2$. This
implies that
\begin{equation}\label{E:x(-1) in step 3A}
a_1=k_0,\quad b_1=k_2,\quad c_1\leq k_0.
\end{equation}
As above, Lemmas~\ref{L:initial conditions for W(Lambda1)},
\ref{L:initial conditions for W(Lambda2)} and \ref{L:initial
conditions for W(Lambda0)} imply that
$$
\begin{aligned}
 x(-1)\sp{\mathcal A_\pi}v_\Lambda
&=x_{\underline
2}(-1)\sp{c_1}x_0(-1)\sp{k_2}x_2(-1)\sp{k_0}\left(v_{\Lambda_{1}}\sp{
k_1} v_{\Lambda_{2}}\sp{ k_2}
v_{\Lambda_{0}}\sp{ k_0}\right)\\
&=C\,x_{\underline 2}(-1)\sp{c_1}\left(v_{\Lambda_{1}}\sp{ k_1}
(x_0(-1)v_{\Lambda_{2}})\sp{ k_2}
(x_2(-1)v_{\Lambda_{0}})\sp{ k_0}\right)\\
&=C'\,v_{\Lambda_{1}}\sp{ k_1} (x_0(-1)v_{\Lambda_{2}})\sp{ k_2}
(x_2(-1)v_{\Lambda_{0}})\sp{ k_0-c_1} (x_{\underline
2}(-1)x_2(-1)v_{\Lambda_{0}})\sp{ c_1}.
\end{aligned}
$$
Let $A\colon V\to V$ be a linear operator
$$
A|_{L(\Lambda_{0})}=[\omega_{\underline 2}]\colon L(\Lambda_{0})\to
L(\Lambda_{2}),\qquad A|_{L(\Lambda_{1})\oplus L(\Lambda_{2})}=0,
$$
and let $A$ act as a derivation on $S(V)$. By
Proposition~\ref{P:coefficients of intertwining operators}
derivation $A$ commutes with the action of $\hat{\mathfrak g}_1$ on
the symmetric algebra $S(V)$. Note that
$Av_{\Lambda_{0}}=[\omega_{\underline
2}]v_{\Lambda_{0}}=w_{\underline 2}$ by
Proposition~\ref{P:coefficients of intertwining operators} and
$x_{\underline 2}(-1)w_{\underline 2}=0$ by Lemma~\ref{L:initial
conditions for W(Lambda2)}, so $A(x_{\underline
2}(-1)x_2(-1)v_{\Lambda_{0}})=0$. Hence by Lemmas~\ref{L:simple
current initial conditions for 0 and 1} and \ref{L:simple current
initial conditions for 2} we have
$$
\begin{aligned}
& \ A\sp{k_0-c_1} x(-1)\sp{\mathcal A_\pi}v_\Lambda\\
&=C''\,v_{\Lambda_{1}}\sp{ k_1} (x_0(-1)v_{\Lambda_{2}})\sp{ k_2}
(x_2(-1)w_{\underline 2})\sp{ k_0-c_1}
(x_{\underline 2}(-1)x_2(-1)v_{\Lambda_{0}})\sp{ c_1}\\
&=C''\,([\omega]v_{\Lambda_{0}})\sp{ k_1}
([\omega]v_{\Lambda_{2}})\sp{ k_2} ([\omega]v_{\Lambda_{2}})\sp{
k_0-c_1}
([\omega]v_{\Lambda_{1}})\sp{ c_1}\\
&=C''\,[\omega]\left(v_{\Lambda_{0}}\sp{ k_1} v_{\Lambda_{2}}\sp{
k_2} v_{\Lambda_{2}}\sp{
k_0-c_1} v_{\Lambda_{1}}\sp{ c_1}\right)\\
&=C''\,[\omega]\left(v_{\Lambda_{0}}\sp{ k_1} v_{\Lambda_{2}}\sp{
k_2+k_0-c_1} v_{\Lambda_{1}}\sp{ c_1}\right).
\end{aligned}
$$
Since $A$ commutes with the action of $\hat{\mathfrak g}_1$ we have
$$
\begin{aligned}
&\ A\sp{k_0-c_1}x(\pi)v_\Lambda
=x(\pi_2)Ax(-1)\sp{\mathcal A_\pi}v_\Lambda\\
&=C''\,x(\pi_2)[\omega]\left(v_{\Lambda_{0}}\sp{ k_1}
v_{\Lambda_{2}}\sp{ k_2+k_0-c_1}
v_{\Lambda_{1}}\sp{ c_1}\right)\\
&=C''\,[\omega]x(\pi_2\sp+)\left(v_{\Lambda_{0}}\sp{ k_1}
v_{\Lambda_{2}}\sp{ k_2+k_0-c_1}
v_{\Lambda_{1}}\sp{ c_1}\right)\\
&=C''\,[\omega]x(\pi_2\sp+)v_{\Lambda'}
\end{aligned}
$$
for some $C''\neq0$. It is clear that ``truncated and shifted by
degree'' monomial $x(\pi_2\sp+)$ satisfies DC, and IC for
$x(\pi_2\sp+)v_{\Lambda'}$ read
$$
a_2\leq k_1=k-b_1-a_1,\quad b_2+a_2\leq k_1+k_2+k_0-c_1,\quad
c_2+b_2\leq k_1+k_2+k_0-c_1=k-c_1.
$$
But these are just three difference condition relations which hold
for $x(\pi)v_\Lambda$:
$$
a_2+b_1+a_1\leq k,\quad b_2+a_2+c_1\leq k,\quad c_2+b_2+c_1\leq k.
$$
Hence we have proved the following:
\begin{lemma}\label{L:proof of independence step 3A}
In the case when \eqref{E:x(-1) in step 3A} holds monomial vector
$$x(\pi_2\sp+)v_{\Lambda'}=(C''\,[\omega])\sp{-1}A\sp{k_0-c_1}x(\pi)v_\Lambda$$
satisfies difference conditions \eqref{E:difference conditions} and
initial conditions \eqref{E:initial conditions}.
\end{lemma}

Assume now that we have a monomial vector $x(\pi)v_\Lambda$ such
that $a_1=k_0$ and $b_1+a_1\leq k_0+k_2$ and  $c_1+b_1=k_0+k_2$.
This implies that
\begin{equation}\label{E:x(-1) in step 3B}
a_1=k_0,\quad b_1\leq k_2,\quad c_1+b_1= k_0+k_2.
\end{equation}
Like before, Lemmas~\ref{L:initial conditions for W(Lambda1)},
\ref{L:initial conditions for W(Lambda2)} and \ref{L:initial
conditions for W(Lambda0)} imply that
$$
\begin{aligned}
&\ x(-1)\sp{\mathcal A_\pi}v_\Lambda\\
&=x_{\underline
2}(-1)\sp{c_1}x_0(-1)\sp{b_1}x_2(-1)\sp{k_0}\left(v_{\Lambda_{1}}\sp{
k_1} v_{\Lambda_{2}}\sp{ k_2}
v_{\Lambda_{0}}\sp{ k_0}\right)\\
&=C\,x_{\underline 2}(-1)\sp{c_1}\left(v_{\Lambda_{1}}\sp{ k_1}
v_{\Lambda_{2}}\sp{ k_2-b_1} (x_0(-1)v_{\Lambda_{2}})\sp{ b_1}
(x_2(-1)v_{\Lambda_{0}})\sp{ k_0}\right)\\
&=C'\,v_{\Lambda_{1}}\sp{ k_1} (x_{\underline
2}(-1)v_{\Lambda_{2}})\sp{ k_2-b_1} (x_0(-1)v_{\Lambda_{2}})\sp{
b_1} (x_{\underline 2}(-1)x_2(-1)v_{\Lambda_{0}})\sp{ k_0}.
\end{aligned}
$$
By Lemmas~\ref{L:simple current initial conditions for 0 and 1} and
\ref{L:simple current initial conditions for 2} we further have
$$
\begin{aligned}
&\ x(-1)\sp{\mathcal A_\pi}v_\Lambda\\
&=C'\,([\omega]v_{\Lambda_{0}})\sp{ k_1} ([\omega]w_{\underline
2})\sp{ k_2-b_1} ([\omega]v_{\Lambda_{2}})\sp{ b_1}
([\omega]v_{\Lambda_{1}})\sp{ k_0}\\
&=C'\,[\omega]\left(v_{\Lambda_{0}}\sp{ k_1} w_{\underline 2}\sp{
k_2-b_1} v_{\Lambda_{2}}\sp{ b_1} v_{\Lambda_{1}}\sp{ k_0}\right).
\end{aligned}
$$
Hence we have
$$
\begin{aligned}
&\ x(\pi)v_\Lambda=x(\pi_2)x(-1)\sp{\mathcal A_\pi}v_\Lambda\\
&=C'\,x(\pi_2)[\omega]\left(v_{\Lambda_{0}}\sp{ k_1} w_{\underline
2}\sp{ k_2-b_1} v_{\Lambda_{2}}\sp{ b_1} v_{\Lambda_{1}}\sp{
k_0}\right)\\
&=C'\,[\omega]x(\pi_2\sp+)\left(v_{\Lambda_{0}}\sp{ k_1}
w_{\underline 2}\sp{ k_2-b_1} v_{\Lambda_{2}}\sp{ b_1}
v_{\Lambda_{1}}\sp{ k_0}\right).
\end{aligned}
$$
Let $B\colon V\to V$ be a linear operator
$$
B|_{L(\Lambda_{2})}=[\omega_{2}]\colon L(\Lambda_{2})\to
L(\Lambda_{1}),\qquad B|_{L(\Lambda_{0})\oplus L(\Lambda_{1})}=0,
$$
and let $B$ act as a derivation on $S(V)$. By
Proposition~\ref{P:coefficients of intertwining operators}
derivation $B$ commutes with the action of $\hat{\mathfrak g}_1$ on
the symmetric algebra $S(V)$ and $Bw_{\underline
2}=[\omega_{2}]w_{\underline 2}=v_{\Lambda_{1}}$ and
$Bv_{\Lambda_{2}}=[\omega_{2}]v_{\Lambda_{2}}=0$. Hence
$$
B\sp{k_2-b_1}\colon v_{\Lambda_{0}}\sp{ k_1} w_{\underline 2}\sp{
k_2-b_1} v_{\Lambda_{2}}\sp{ b_1} v_{\Lambda_{1}}\sp{ k_0}\mapsto
C''\,v_{\Lambda'}= C''\,v_{\Lambda_{0}}\sp{ k_1} v_{\Lambda_{1}}\sp{
k_0+k_2-b_1} v_{\Lambda_{2}}\sp{ b_1}.
$$
\begin{lemma}\label{L:proof of independence step 3B}
In the case when \eqref{E:x(-1) in step 3B} holds monomial vector
$$x(\pi_2\sp+)v_{\Lambda'}=B\sp{k_2-b_1}(C''C'\,[\omega])\sp{-1}x(\pi)v_\Lambda$$
satisfies difference conditions \eqref{E:difference conditions} and
initial conditions \eqref{E:initial conditions}.
\end{lemma}
\begin{proof}
It is clear that ``truncated and shifted by degree'' monomial
$x(\pi_2\sp+)$ satisfies DC, and IC for $x(\pi_2\sp+)v_{\Lambda'}$
read
$$
a_2\leq k_1=k-c_1-b_1,\quad b_2+a_2\leq k_1+b_1,\quad c_2+b_2\leq
k_1+b_1=k_1+k_0+k_2-c_1=k-c_1.
$$
But these are just three difference condition relations which hold
for $x(\pi)v_\Lambda$:
$$
a_2+c_1+b_1\leq k,\quad b_2+a_2+c_1\leq k,\quad c_2+b_2+c_1\leq k.
$$
\end{proof}

Now we proceed with the proof of linear independence. As already
noted, in the linear combination \eqref{E:induction step 2'} we need
to consider only
$$
\begin{aligned}
0&=\sum_{\begin{subarray}{c}a_1=k_0 \\b_1+a_1=k_0+k_2>c_1+b_1\\
\text{or}\\a_1=k_0 \\b_1+a_1\leq
k_0+k_2=c_1+b_1\end{subarray}}C_{\dots c_1b_1a_1}\dots x_{\underline
2}(-1)\sp{c_1}x_0(-1)\sp{b_1}x_2(-1)\sp{k_0}\left(v_{\Lambda_{1}}\sp{
k_1} v_{\Lambda_{2}}\sp{ k_2}
v_{\Lambda_{0}}\sp{ k_0}\right)\\
&=\sum_{\begin{subarray}{c}c_1<k_0\end{subarray}}C_{\dots
c_1k_2k_0}\dots x_{\underline
2}(-1)\sp{c_1}x_0(-1)\sp{k_2}x_2(-1)\sp{k_0}\left(v_{\Lambda_{1}}\sp{
k_1} v_{\Lambda_{2}}\sp{ k_2}
v_{\Lambda_{0}}\sp{ k_0}\right)\\
&+\sum_{\begin{subarray}{c}b_1\leq k_2\\
c_1+b_1=k_0+k_2\end{subarray}}C_{\dots c_1b_1k_0}\dots x_{\underline
2}(-1)\sp{c_1}x_0(-1)\sp{b_1}x_2(-1)\sp{k_0}\left(v_{\Lambda_{1}}\sp{
k_1} v_{\Lambda_{2}}\sp{ k_2} v_{\Lambda_{0}}\sp{ k_0}\right).
\end{aligned}
$$
Note that in the first sum we have vectors of the form
\begin{equation}\label{E:vectors in the sum A}
x(\pi_2)\left(v_{\Lambda_{1}}\sp{ k_1} (x_0(-1)v_{\Lambda_{2}})\sp{
k_2} (x_2(-1)v_{\Lambda_{0}})\sp{ k_0-c_1} (x_{\underline
2}(-1)x_2(-1)v_{\Lambda_{0}})\sp{ c_1}\right)
\end{equation}
for $k_0-c_1=1, \dots, k_0$, and that in the second sum we have
vectors of the form
\begin{equation}\label{E:vectors in the sum B}
x(\pi_2)\left(v_{\Lambda_{1}}\sp{ k_1} (x_{\underline
2}(-1)v_{\Lambda_{2}})\sp{ k_2-b_1} (x_0(-1)v_{\Lambda_{2}})\sp{
b_1} (x_{\underline 2}(-1)x_2(-1)v_{\Lambda_{0}})\sp{ k_0}\right)
\end{equation}
for $k_2-b_1=0, \dots, k_2$. In particular in \eqref{E:vectors in
the sum A} we see a factor
$$
(x_2(-1)v_{\Lambda_{0}})\sp{ k_0-c_1}(x_{\underline
2}(-1)x_2(-1)v_{\Lambda_{0}})\sp{ c_1}\quad\text{for}\quad
c_1=0,\dots, k_0-1
$$
and in \eqref{E:vectors in the sum B} we see a factor
$(x_{\underline 2}(-1)x_2(-1)v_{\Lambda_{0}})\sp{ k_0}$. Hence the
operator $A\sp{k_0}$ annihilate all these terms except ones with
$c_1=0$ and the action on linear combination \eqref{E:induction step
2'} gives
$$
0=A\sp{k_0}\sum c_\pi x(\pi)v_\Lambda=[\omega]\sum_{\mathcal
A_\pi=(0,k_2,k_0)}c_\pi C''\,x(\pi_2\sp+)v_{\Lambda'}.
$$
Now Lemma~\ref{L:proof of independence step 3A} and the induction
hypothesis imply that $c_\pi=0$ whenever $A_\pi=(0,k_2,k_0)$. In
turn this implies that in the first sum of \eqref{E:induction step
2'} it is enough to consider vectors \eqref{E:vectors in the sum A}
for $c_1=1,\dots, k_0-1$. Then we apply operator $A\sp{k_0-1}$ which
annihilate all these terms except ones with $c_1=1$ and the action
on linear combination \eqref{E:induction step 2'} gives
$$
0=A\sp{k_0}\sum c_\pi x(\pi)v_\Lambda=[\omega]\sum_{\mathcal
A_\pi=(1,k_2,k_0)}c_\pi C''\,x(\pi_2\sp+)v_{\Lambda'}.
$$
Now Lemma~\ref{L:proof of independence step 3A} and the induction
hypothesis imply that $c_\pi=0$ whenever $A_\pi=(1,k_2,k_0)$. By
proceeding in this way we see that all the coefficients
$c_\pi=C_{\dots c_1b_1a_1}$ for $c_1<k_0$ in the first sum are equal
to zero.

So we are left with the second sum
\begin{equation}\label{E:the second sum in step 3B}
\sum c_\pi x(\pi)v_\Lambda=
[\omega]\sum_{\begin{subarray}{c}b_1\leq k_2\\
c_1+b_1=k_0+k_2\end{subarray}}c_\pi
C'\,x(\pi_2\sp+)\left(v_{\Lambda_{0}}\sp{ k_1} w_{\underline 2}\sp{
k_2-b_1} v_{\Lambda_{2}}\sp{ b_1} v_{\Lambda_{1}}\sp{ k_0}\right)=0.
\end{equation}
This implies
\begin{equation}\label{E:last step in step 3B}
\sum_{\begin{subarray}{c}b_1\leq k_2\\
c_1+b_1=k_0+k_2\end{subarray}}c_\pi
C'\,x(\pi_2\sp+)\left(v_{\Lambda_{0}}\sp{ k_1} w_{\underline 2}\sp{
k_2-b_1} v_{\Lambda_{2}}\sp{ b_1} v_{\Lambda_{1}}\sp{ k_0}\right)=0.
\end{equation}
In \eqref{E:last step in step 3B} we see factors
$$
w_{\underline 2}\sp{ k_2-b_1} v_{\Lambda_{2}}\sp{
b_1}\qquad\text{for}\quad b_1=0,\dots, k_2.
$$
The operator $B\sp{k_0}$ will annihilate all these terms except ones
with $b_1=0$ and the action on linear combination \eqref{E:last step
in step 3B} gives
$$
\sum_{\begin{subarray}{c}b_1=0 \\
c_1+b_1=k_0+k_2\end{subarray}}c_\pi
C'C''\,x(\pi_2\sp+)\left(v_{\Lambda_{0}}\sp{ k_1}
 v_{\Lambda_{1}}\sp{ k_0+k_2-b_1}v_{\Lambda_{2}}\sp{ b_1}\right)=0.
$$
Now Lemma~\ref{L:proof of independence step 3B} and the induction
hypothesis imply that $c_\pi=0$ whenever $A_\pi=(k_0+k_2,0,k_0)$. In
turn this implies that in \eqref{E:last step in step 3B} it is
enough to consider vectors for $b_1=1,\dots, k_0$. So next we apply
$B\sp{k_0-1}$ and conclude that $c_\pi=0$ whenever
$A_\pi=(k_0+k_2-1,1,k_0)$. By proceeding in this way we see that all
the coefficients $c_\pi$ in the second sum \eqref{E:the second sum
in step 3B} are equal to zero and our proof of linear independence
is complete.

\section{Vertex operator formula}\label{S:Vertex operator formula}

For a root $\alpha$ we denote by $\alpha\sp\vee\in\mathfrak h$ a
dual root, $\alpha\sp\vee=2\alpha/\langle\alpha,\alpha\rangle$. In
this section for each $\alpha \in R$ we choose $x_\alpha \in
\mathfrak g_\alpha$ so that $[x_\alpha, x_{-\alpha}] = -\alpha^\vee$
and define on $L(\Lambda)$ a ``Weyl group translation'' operator
$e_\alpha$ by
$$
\begin{aligned}
&s_\alpha = \exp x_\alpha(0) \exp x_{-\alpha}(0) \exp x_\alpha(0), \\
&s_{\delta-\alpha} = \exp x_{-\alpha}(1)\exp x_\alpha(-1) \exp x_{-\alpha}(1),\\
&e_\alpha = s_{\delta-\alpha} s_\alpha.
\end{aligned}$$
Then on a standard $\hat{\mathfrak g}$-module $L(\Lambda)$ we have
\begin{equation}\label{E:Weyl translations normalize root
vectors}
\begin{aligned}
&e_\alpha c \, e_\alpha^{-1} = c , \\
&e_\alpha d\, e^{-1}_\alpha = d + \alpha^\vee - \tfrac 12 \langle
\alpha^\vee,
\alpha^\vee \rangle c , \\
&e_\alpha h \, e^{-1}_\alpha = h - \langle \alpha^\vee, h \rangle c ,\\
&e_\alpha h(j)\, e^{-1}_\alpha = h(j)\qquad \text{for}\  j \ne 0,\\
&e_\alpha x_\gamma (j)\, e^{-1}_\alpha = (-1)^{\gamma(\alpha^\vee)}
x_\gamma (j - \gamma (\alpha^{\vee}))
\end{aligned}
\end{equation}
for $h \in \mathfrak h$, $\gamma \in R$ and $j \in \Bbb Z$. These
formulas are a consequence of the adjoint action of the group
element $e_\alpha$ on $\hat{\mathfrak g}$. The map $\alpha^\vee
\mapsto e_{\alpha}$ extends from the dual root system $R^\vee$ to a
projective representation of the root lattice $Q(R^\vee)$ (cf.
\cite{FK}, \cite{Fr}, \cite{K}).

Let $\alpha \in R$. Then $\widehat{\mathfrak {sl}}_2 (\alpha)$
defined by \eqref{E:definition of sl2(alpha)} is of the type
$A^{(1)}_1$ with the canonical central element
\begin{equation*}
c_\alpha = \langle x_\alpha, - x_{-\alpha}\rangle c = 2c/
\langle\alpha,\alpha\rangle .
\end{equation*}
For a standard $\hat{\mathfrak g}$-module $L(\Lambda)$ of level
$\Lambda(c) = k$, the restriction to $\widehat{\mathfrak
{sl}}_2(\alpha)$ is of level $k_\alpha=\Lambda(c_\alpha)=k$ if
$\langle\alpha,\alpha\rangle = 2$ (i.e. if $\alpha$ is a long root)
and of level $k_\alpha=2k$ if $\langle\alpha,\alpha\rangle = 1$
(i.e. if $\alpha$ is a short root). Recall that $z\,x_\alpha(z)=\sum
x_\alpha(n)z^{-n}$ is a formal Laurent series in an indeterminate
$z$ with coefficients in $\text{End\,}(L(\Lambda))$. We also define
a formal Laurent series $z^{c_\alpha+\alpha^\vee}$ by
$$
z^{c_\alpha+\alpha^\vee} v_\mu = v_\mu z^{k_\alpha+\mu(\alpha^\vee)}
$$
whenever $v_\mu \in L(\Lambda)$ is a vector of $\mathfrak h$-weight
$\mu$. Set
\begin{equation*}
E^\pm(\alpha,z) = \exp \biggl(\sum_{i>0} \alpha^\vee (\pm i) z^{\mp
i}\big/ (\pm i)\biggr).
\end{equation*}
Since
$$
x_\alpha(z)^{k_\alpha+1} = 0\quad \text{on}\quad L(\Lambda),
$$
the exponential \ $\exp (z\,x_\alpha(z))=\exp\big(\sum
x_\alpha(n)z^{-n}\big)$ is well defined, and we have a
generalization of the Frenkel-Kac vertex operator formula (cf. [LP,
Theorem 5.6], [P1, Theorem 6.4] or [P2, Section 3]) for all standard
modules:
\begin{equation}\label{E:vertex operator formula for all levels}
\exp (z\,x_\alpha (z)) = E^-(-\alpha,z)\exp(- z\,x_{-\alpha}(z))
E^+(-\alpha,z)\, e_\alpha\, z^{c_\alpha+\alpha^\vee}.
\end{equation}
By \eqref{E:Weyl translations normalize root vectors} the $\mathfrak
h$-weight components of the vertex operator formula \eqref{E:vertex
operator formula for all levels} on level $k$ module $L(\Lambda)$
give relations
\begin{equation}\label{3.9}
\frac{1}{p!}\,(z\,x_\alpha(z))^{p} =\frac{1}{q!}
E^-(-\alpha,z)\,(-z\,x_{-\alpha}(z))^{q}\, E^+(-\alpha,z)\, e_\alpha
\,z^{k_\alpha+\alpha^\vee}
\end{equation}
for $p,q\geq 0$, \ $p+q=k_\alpha$. In level $k=1$ case for a long
root $\alpha$ and $p=1,0$ we have
\begin{equation}\label{E:Frenkel-Kac formulas0}
\begin{aligned}
z\,x_\alpha(z) &= E^-(-\alpha,z)\, E^+(-\alpha,z)\, e_\alpha \,z^{1+\alpha},\\
1 &= E^-(-\alpha,z)\,(-z\,x_{-\alpha}(z))\, E^+(-\alpha,z)\,
e_\alpha\,z^{1+\alpha}.
\end{aligned}
\end{equation}
Since in this case $e_{\alpha}e_{-\alpha}=-1$, relations
\eqref{E:Frenkel-Kac formulas0} are simply the Frenkel-Kac vertex
operator formulas
\begin{equation*}
\begin{aligned}
x_\alpha(z)&= E^-(-\alpha,z)\, E^+(-\alpha,z)\, e_\alpha
\,z^{\alpha},\\
x_{-\alpha}(z)&= E^-(\alpha,z)\, E^+(\alpha,z)\, e_{-\alpha}
\,z^{-\alpha}
\end{aligned}
\end{equation*}
(see \cite{FK}, \cite{Fr}). In fact, the Frenkel-Kac vertex operator
formulas for level $1$ standard $\widehat{\mathfrak {sl}}_2
(\alpha)$-modules imply $x_{\alpha}(z)\sp2=x_{-\alpha}(z)\sp2=0$ and
the relation \eqref{E:vertex operator formula for all levels}, and
for higher level $k$ modules we prove \eqref{E:vertex operator
formula for all levels} by simply applying the ``exponentials of Lie
algebra elements'' on both sides of \eqref{E:vertex operator formula
for all levels} to tensor product of $k$ copies of level $1$
modules.

Let us denote by $\langle\,e_\alpha\,;\,\alpha\in\Gamma\,\rangle$ a
group of operators on $L(\Lambda)$ generated by all operators
$e_\alpha$, $\alpha\in\Gamma$. Then we have:

\begin{lemma}\label{L:lemma 1 on spanning of standard module}
$L(\Lambda)=\langle\,e_\alpha\,;\,\alpha\in\Gamma\,\rangle\,
U(\hat{\mathfrak g}_1)\,v_\Lambda$.
\end{lemma}
\begin{proof} First notice that the Lie algebra $\mathfrak g$ is
generated by $\mathfrak g_1 \cup \mathfrak g_{-1}$. In particular
$\text{span\,} \Gamma^\vee=\mathfrak h$. Similarly $[\hat{\mathfrak
g},\hat{\mathfrak g}]$ is generated by $\hat{\mathfrak g}_1 \cup
\hat{\mathfrak g}_{-1}$, so we have
$$
L(\Lambda)= \{\,x_1\dots x_s \,v_\Lambda\mid s\geqslant 0, \ x_i\in
\hat{\mathfrak g}_1 \cup \hat{\mathfrak g}_{-1}\}.
$$
By using vertex operator formula \eqref{3.9} for $\alpha\in
(-\Gamma)$ and $p=1$ we may replace each $x_i\in \hat{\mathfrak
g}_{-1}$ by a product of elements from
$$
\{\,e_\alpha\mid\alpha\in (-\Gamma)\} \cup \mathfrak s \cup
\hat{\mathfrak g}_1,
$$
where $\mathfrak s$ denotes the Heisenberg subalgebra
$$
\mathfrak s = \sum_{j\in\mathbb Z\backslash\{0\}}\, \mathfrak h
\otimes t^j + \mathbb C c,\qquad \mathfrak s_- = \sum_{j<0}\,
\mathfrak h \otimes t^j.
$$
Since both group elements $e_\alpha$ and Lie algebra elements from
the Heisenberg subalgebra $\mathfrak s$ normalize $\hat{\mathfrak
g}_1$, we get
$$
L(\Lambda_0) = \langle\,e_\alpha\,;\,\alpha\in\Gamma\,\rangle\,
U(\hat{\mathfrak g}_1)\, U(\mathfrak s_-)\,v_\Lambda.
$$
Now notice that $U(\mathfrak s_-)$ is generated by the coefficients
of $E^-(-\alpha,z)$ for $\alpha\in\Gamma$. So by using vertex
operator formula \eqref{3.9} for $\alpha\in \Gamma$ and $q=0$ we may
replace elements in $U(\mathfrak s_-)v_\Lambda$ by elements in
$\langle\,e_\alpha\,;\,\alpha\in\Gamma\,\rangle \,U(\hat{\mathfrak
g}_1)\,v_\Lambda$.
\end{proof}

As in [P2, Section 5] we set
\begin{equation}\label{E: Weyl group translation operator e}
e=e_{\varepsilon_1-\varepsilon_2}e_{\varepsilon_1}e_{\varepsilon_1+\varepsilon_2}
=\prod_{\alpha\in\Gamma}\,e_\alpha.
\end{equation}

\begin{proposition}\label{P:spanning of standard module}
Let $L(\Lambda)_\mu$ be a weight subspace of $L(\Lambda).$ Then
there exists an integer $m_0$ such that for any fixed $m\leqslant
m_0$ the set of vectors
$$
e^m\,x_{\beta_1}(j_1)\dots x_{\beta_s}(j_s)\,v_\Lambda \in
L(\Lambda)_\mu,
$$
where $s\geqslant 0$, $\beta_1,\dots,\beta_s \in \Gamma$,
$j_1,\dots,j_s \in \Bbb Z$, is a spanning set of $L(\Lambda)_\mu$.
In particular
$$
L(\Lambda)=\langle\,e\,\rangle\, W(\Lambda).
$$
\end{proposition}
\begin{proof} Since $\dim L(\Lambda)_\mu < \infty$, by
Lemma~\ref{L:lemma 1 on spanning of standard module} we may choose a
finite spanning set of vectors of the form
$$
(\prod_{\alpha\in\Gamma}\,e_\alpha)^m\,\prod_{\alpha\in\Gamma}\,
e_\alpha^{p_\alpha}\, x_{\beta_1}(j_1)\dots
x_{\beta_r}(j_r)\,v_\Lambda,
$$
$r\geqslant 0$, $x_{\beta_i}(j_i) \in \hat{\mathfrak g}_1$, $ m$
fixed for all vectors. Clearly there exists $m_0$ such that if we
choose $m\leqslant m_0,$ then all $p_\alpha\geqslant 0$ for all
vectors. Since $e_\alpha$ normalize $\hat{\mathfrak g}_1$, we have a
spanning set of vectors of the form
$$
e^m\, x_{\beta_1}(j'_1)\dots x_{\beta_r}(j'_r)\,
\prod_{\alpha\in\Gamma} \,e_\alpha^{p_\alpha}\, v_\Lambda.
$$
Now in a finite number of steps we replace each $e_\alpha v_\Lambda$
by an element in $U(\hat{\mathfrak g}_1)v_\Lambda$ by using
coefficients of $z^{k_\alpha+\Lambda(\alpha\sp\vee)}$ in vertex
operator formula \eqref{3.9} for $\alpha\in \Gamma$ and $q=0$.
\end{proof}

\section{Bases consisting of semi-infinite monomials}\label{S:semi-infinite monomials}

In \eqref{E:Lambda to Lambda star} we have set
$\Lambda\sp*=k_1\Lambda_0+k_0\Lambda_1+k_2\Lambda_2
\quad\text{for}\quad
\Lambda=k_0\Lambda_0+k_1\Lambda_1+k_2\Lambda_2$. Note that
$\Lambda\sp{**}=\Lambda$. The relation \eqref{E:v(Lambda to Lambda
star)} applied twice together with Lemma~\ref{L:commutation of omega
and monomials} gives

\begin{equation}\label{E:v(Lambda to Lambda star twice)}
\begin{aligned}
&\
[\omega]\sp2v_{\Lambda}=[\omega]\sp2\left(v_{\Lambda_0}\sp{\otimes
k_0}\otimes v_{\Lambda_1}\sp{\otimes k_1}\otimes
v_{\Lambda_{2}}\sp{\otimes
k_2}\right)\\
&=C'\,[\omega]x_{\underline{2}}(-1)\sp{k_1}x_0(-1)\sp{k_2}x_2(-1)\sp{k_1}\left(v_{\Lambda_1}\sp{\otimes
k_0}\otimes v_{\Lambda_0}\sp{\otimes k_1}\otimes
v_{\Lambda_{2}}\sp{\otimes k_2}\right)\\
&=C'\,x_{\underline{2}}(-2)\sp{k_1}x_0(-2)\sp{k_2}x_2(-2)\sp{k_1}[\omega]\left(v_{\Lambda_1}\sp{\otimes
k_0}\otimes v_{\Lambda_0}\sp{\otimes k_1}\otimes
v_{\Lambda_{2}}\sp{\otimes k_2}\right)\\
&= C\,x_{\underline{2}}(-2)\sp{k_1}x_0(-2)\sp{k_2}x_2(-2)\sp{k_1}
x_{\underline{2}}(-1)\sp{k_0}x_0(-1)\sp{k_2}x_2(-1)\sp{k_0}v_{\Lambda}
\end{aligned}
\end{equation}
for some $C=C_\Lambda\neq0$. If we set
\begin{equation*}
x(\kappa_\Lambda)=x_{\underline{2}}(-2)\sp{k_1}x_0(-2)\sp{k_2}x_2(-2)\sp{k_1}
x_{\underline{2}}(-1)\sp{k_0}x_0(-1)\sp{k_2}x_2(-1)\sp{k_0},
\end{equation*}
then \eqref{E:v(Lambda to Lambda star twice)} reads
\begin{equation}\label{E: omega2 v(Lambda) to v(Lambda)}
[\omega]\sp2v_{\Lambda}=C_\Lambda x(\kappa_\Lambda)v_{\Lambda}.
\end{equation}
This relation and Lemma~\ref{L:commutation of omega and monomials}
imply
\begin{equation*}
[\omega]\sp2\colon L(\Lambda)\to L(\Lambda)\quad\text{and}\quad
[\omega]\sp2\colon W(\Lambda)\to W(\Lambda).
\end{equation*}

\begin{lemma}\label{L:Weyl translation and simple current operator}
$e=C\,[\omega]\sp4$ \ for some \ $C\neq0$.
\end{lemma}
\begin{proof}
Since $(\varepsilon_1-\varepsilon_2)\sp\vee+\varepsilon_1\sp\vee+
(\varepsilon_1+\varepsilon_2)\sp\vee=4\varepsilon_1$, relations
\eqref{E:Weyl translations normalize root vectors} and
\eqref{commutation of omega and x alpha} imply
\begin{equation*} e\, x_{\pm\gamma}(j)e\sp{-1}= x_\gamma(j\mp 4)\quad\text{and}\quad
[\omega]\sp4 x_{\pm\gamma}(j)[\omega]\sp{-4}= x_\gamma(j\mp 4)
\end{equation*}
for all $\gamma\in\Gamma$. So $e\,[\omega]\sp{-4}$ commutes with the
action of $\hat{\mathfrak g}$ and must be proportional to the
identity operator on $L(\Lambda)$.
\end{proof}

\begin{remark}\label{R:omega is a 4th root of e}
Roughly speaking, the above lemma states that the simple current
operator $[\omega]$ is a ``fourth root'' of inner automorphism $e$.
By choosing $e=\prod_{\alpha\in\Gamma}\,e_\alpha$ we follow the
notation in \cite{P2}, but our arguments would work in the same way
if we have chosen inner automorphism $e$ to be
$$
e_{\varepsilon_1-\varepsilon_2}e_{\varepsilon_1+\varepsilon_2}=
C'e_{\varepsilon_1}=C[\omega]\sp2
$$
for some $C',C\neq0$.\end{remark}

\begin{theorem}\label{T:semi infinite monomials basis} Let
$L(\Lambda)_\mu$ be a weight subspace of a level $k$ standard
$B_2\sp{(1)}$-module $L(\Lambda)$. Then there exists an integer
$m_0$ such that for any fixed $m\leq m_0$ the set of vectors
$$
C_\Lambda\sp{-m}\,[\omega]\sp{2m} x(\pi)v_\Lambda\in L(\Lambda)_\mu
$$
such that monomial vectors $x(\pi)v_\Lambda\in W(\Lambda)$ satisfy
difference conditions \eqref{E:difference conditions} and initial
conditions \eqref{E:initial conditions}
 is a basis of $L(\Lambda)_\mu$. Moreover, for two choices of $m_1,
 m_2\leq m_0$ the corresponding two bases are equal.
\end{theorem}
\begin{proof}
By Proposition~\ref{P:spanning of standard module} vectors of the
form
$$
e\sp mx(\pi)v_\Lambda\in L(\Lambda)_\mu
$$
span $L(\Lambda)_\mu$ for a given small enough $m$, so by
Theorem~\ref{T:the main theorem} monomial vectors satisfying DC and
IC will form a basis. By Lemma~\ref{L:Weyl translation and simple
current operator} we can replace $e\sp{m}$ with $[\omega]\sp{4m}$.
Note that by Lemma~\ref{L:commutation of omega and monomials}, the
notation from Remark~\ref{R:notation for shifted partitions} and
\eqref{E: omega2 v(Lambda) to v(Lambda)}
\begin{equation}\label{E:sequence of basis vectors}
C_\Lambda\sp{-m}[\omega]\sp{2m}x(\pi)v_\Lambda
\!=\![\omega]\sp{2m-2}x(\pi\sp{-2})[\omega]\sp{2}v_\Lambda
\!=\!C_\Lambda\sp{-m+1}
[\omega]\sp{2(m-1)}x(\pi\sp{-2})x(\kappa_\Lambda)v_\Lambda
\end{equation}
and the monomial vector $x(\pi)v_\Lambda$ satisfies DC and IC if and
only if the monomial vector
$x(\pi\sp{-2})x(\kappa_\Lambda)v_\Lambda$ satisfies DC and IC. We
can iterate this process
$$
C_\Lambda\sp{-m}\,[\omega]\sp{2m}x(\pi)v_\Lambda =\dots
=C_\Lambda\sp{-m+2}\,[\omega]\sp{2(m-2)}x(\pi\sp{-4})x(\kappa_\Lambda\sp{-2})
x(\kappa_\Lambda)v_\Lambda=\dots\,.
$$
Hence for different choices of integers $m, m-1, m-2, \dots$  we
always get the same basis vector \eqref{E:sequence of basis
vectors}, only written in a different way.
\end{proof}

\begin{remark} One may think of Theorem~\ref{T:semi infinite
monomials basis} as a vertex operator construction for an arbitrary
standard $\hat{\mathfrak g}$-module $L(\Lambda)$. And while the
basis constructed by using an inner automorphism and three
commutative currents is relatively simple, the action of
$\hat{\mathfrak g}$ is given by a complicated implicit use of vertex
operator formula \eqref{E:vertex operator formula for all levels}.
\end{remark}

In the level $k=1$ case linear independence in this theorem is
proved in \cite{P2} for the basic representation $L(\Lambda_0)$ by
writing basis elements as semi-infinite monomials and then
``counting'' them by using crystal base character formula
\cite{KKMMNN}. Such semi-infinite monomials interpretation is
possible for all standard $B_2\sp{(1)}$-modules, like in \cite{FS}
for $A_1\sp{(1)}$: for fixed $\Lambda$ and $m\in \mathbb Z$ set
$$
v_{-m}=C_\Lambda\sp m\,[\omega]\sp{-2m}v_\Lambda.
$$
From Lemma~\ref{L: degree for level k omega} we see that $\mathfrak
h$-weight of $v_{-m}$ is $\Lambda|\mathfrak h-2mk\varepsilon_1$ and
degree  of $v_{-m}$ is $2m\Lambda(\varepsilon_1)-2m\sp2k$. By using
Lemma~\ref{L:commutation of omega and monomials} and \eqref{E:
omega2 v(Lambda) to v(Lambda)} as in \eqref{E:sequence of basis
vectors} we get
\begin{equation}\label{E:sequence of vectors v-m}
v_{-m}= x(\kappa_\Lambda\sp{+2(m+1)})v_{-m-1}=
x(\kappa_\Lambda\sp{+2(m+1)})x(\kappa_\Lambda\sp{+2(m+2)})v_{-m-2}=\dots\,.
\end{equation}
So by ``taking a limit'' we see that the vector $v_{-m}$ can be
represented by a semi-infinite quasi-periodic monomial
$$
v_{-m}\sim x(\kappa_\Lambda\sp{+2(m+1)})\cdot
x(\kappa_\Lambda\sp{+2(m+2)})\cdot\dots\,= \prod_{p=1}\sp\infty
x(\kappa_\Lambda\sp{+2(m+p)})\,,
$$
or written in more detail
$$
\begin{aligned}
&v_{-m}\sim
x_{\underline{2}}(2m)\sp{k_1}x_0(2m)\sp{k_2}x_2(2m)\sp{k_1}
x_{\underline{2}}(2m+1)\sp{k_0}x_0(2m+1)\sp{k_2}x_2(2m+1)\sp{k_0}\,\cdot\\
&x_{\underline{2}}(2m+2)\sp{k_1}x_0(2m+2)\sp{k_2}x_2(2m+2)\sp{k_1}
x_{\underline{2}}(2m+3)\sp{k_0}x_0(2m+3)\sp{k_2}x_2(2m+3)\sp{k_0}\dots
\end{aligned}
$$
Now we can write basis elements of $L(\Lambda)_\mu$ given by
Theorem~\ref{T:semi infinite monomials basis} as
\begin{equation}\label{E: basis vector}
C_\Lambda\sp{m}[\omega]\sp{-2m}x(\pi)v_\Lambda=x(\pi\sp{+2m})\,C_\Lambda\sp{m}[\omega]\sp{-2m}v_\Lambda
=x(\pi\sp{+2m})v_{-m}.
\end{equation}
Then \eqref{E:sequence of vectors v-m} implies
$$
\begin{aligned}
&x(\pi\sp{+2m})v_{-m}\\
&=x(\pi\sp{+2m})x(\kappa_\Lambda\sp{+2(m+1)})v_{-m-1}\\
&=x(\pi\sp{+2m})x(\kappa_\Lambda\sp{+2(m+1)})x(\kappa_\Lambda\sp{+2(m+2)})v_{-m-2}\\
&\quad\vdots
\end{aligned}
$$
and we see that our basis vector \eqref{E: basis vector} can be
represented by a semi-infinite monomial with quasi-periodic tail
$$
x(\pi\sp{+2m})v_{-m}\sim
x(\pi\sp{+2m})x(\kappa_\Lambda\sp{+2(m+1)})x(\kappa_\Lambda\sp{+2(m+2)})
x(\kappa_\Lambda\sp{+2(m+3)})\dots\,.
$$
Hence we have:

\begin{corollary}\label{C: bases as semi-infinite monomials}
We can parametrize a basis of level $k$ standard
$B_2\sp{(1)}$-module $L(\Lambda)$,
$$\Lambda=\Lambda_0k_0+\Lambda_1k_1+\Lambda_2k_2\,,\qquad k=k_0+k_1+k_2\,,$$
by semi-infinite monomials
\begin{equation*}
\prod_{j\in\mathbb
Z}x_{\underline{2}}(-j)\sp{c_j}x_0(-j)\sp{b_j}x_2(-j)\sp{a_j}\,,\qquad
c_j=b_j=a_j=0\quad\text{for}\quad -j\ll 0,
\end{equation*}
with quasi-periodic tail  with the period of length $6$
\begin{equation*}
\begin{aligned}
(\dots, \ &c_{-2n},\ b_{-2n},\ a_{-2n},\ c_{-2n-1},\ b_{-2n-1},\
a_{-2n-1},\ \dots)\\=(\dots,& \ k_1,\ k_2,\ k_1,\ k_0,\ k_2,\ k_0,\
\dots)
\end{aligned}
\end{equation*}
for $n\gg 0$, satisfying for all $j\in\mathbb Z$ difference
conditions
\begin{equation*}
\begin{aligned}
c_{j+1}+b_{j+1}+c_{j}&\leq k,\\
b_{j+1}+a_{j+1}+c_{j}&\leq k,\\
a_{j+1}+c_{j}+b_{j}&\leq k,\\
a_{j+1}+b_{j}+a_{j}&\leq k.
\end{aligned}
\end{equation*}
\end{corollary}

Note that for semi-infinite monomials the initial conditions follow
from the form of quasi-periodic tail and the difference conditions.

\section{Presentation of $W(\Lambda)$}

\begin{theorem}\label{T:presentation of W(Lambda)}
Let $\Lambda=k_0\Lambda_0+k_1\Lambda_1+k_2\Lambda_2$ and
$k=k_0+k_1+k_2$. Let $$\mathcal P=\mathbb
C[x_{\underline{2}}(j),x_0(j),x_2(j)\mid j\leq-1]$$ and let
$\mathcal I_\Lambda$ be the ideal in the polynomial algebra
$\mathcal P$ generated by the set of polynomials
$$\begin{aligned}
&\bigcup_{\substack{n\leq-k-1}}U(\mathfrak
g_0)\cdot\Big(\sum_{\substack{j_1,\dots,j_{k+1}\leq-1\\j_1+\dots+j_{k+1}=n}}x_2(j_1)\dots
x_2(j_{k+1})\Big)\\
&\quad\bigcup \{x_2(-1)\sp{k_0+1}\}\bigcup U(\mathfrak g_0)\cdot
x_2(-1)\sp{k_0+k_2+1}\,,
\end{aligned}
$$
where $\cdot$ denotes the adjoint action of $\mathfrak g_0$ on
$\mathcal P$. Then as vector spaces
$$
W(\Lambda)\cong \mathcal P/\mathcal I_\Lambda.
$$
\end{theorem}
\begin{proof} Since $\mathcal P\subset S(\hat{\mathfrak g}_1)=U(\hat{\mathfrak
g}_1)$,
we have a linear map
$$
f\colon \mathcal P\to W(\Lambda),\quad f\colon x(\pi)\mapsto
x(\pi)v_\Lambda.
$$
Since $x(j)v_\Lambda=0$ for $x\in\mathfrak g_1$ and $j\geq0$,
relations $U(\mathfrak g_0)\cdot x_\theta(z)\sp{k+1}=0$ on
$L(\Lambda)$ imply
$$
\bigcup_{\substack{n\leq-k-1}}U(\mathfrak
g_0)\cdot\Big(\sum_{\substack{j_1,\dots,j_{k+1}\leq-1\\j_1+\dots+j_{k+1}=n}}x_2(j_1)\dots
x_2(j_{k+1})\Big)\subset \ker f.
$$
From the proof of Lemma~\ref{L:spanning for W(Lambda)} we see that
$$
\{x_2(-1)\sp{k_0+1}\}\bigcup U(\mathfrak g_0)\cdot
x_2(-1)\sp{k_0+k_2+1}\subset \ker f.
$$
Hence we have a surjective linear map
$$
g\colon \mathcal P/\mathcal I_\Lambda\to W(\Lambda).
$$
On the quotient $\mathcal P/\mathcal I_\Lambda$ we have relations
$$
U(\mathfrak
g_0)\cdot\Big(\sum_{\substack{j_1,\dots,j_{k+1}\leq-1\\j_1+\dots+j_{k+1}=n}}x_2(j_1)\dots
x_2(j_{k+1})\Big)=0\quad\text{for all}\quad n\leq-k-1
$$
and
$$
x_2(-1)\sp{k_0+1}=0 \quad\text{and}\quad U(\mathfrak g_0)\cdot
x_2(-1)\sp{k_0+k_2+1}=0.
$$
As in the proof of Lemma~\ref{L:spanning for W(Lambda)} we see that
monomials $x(\pi)\in \mathcal P$ satisfying DC and IC span the
quotient $\mathcal P/\mathcal I_\Lambda$. Since $g$ maps this
spanning set to a basis of $W(\Lambda)$, monomials $x(\pi)\in
\mathcal P$ satisfying DC and IC are a basis of $\mathcal P/\mathcal
I_\Lambda$ and $g$ is an isomorphism.
\end{proof}

\section{A connection with monomial bases of standard
$A_1\sp{(1)}$-modules}

Let now $\mathfrak g={\mathfrak sl}(2,\mathbb C)$ with the standard
basis $e, h, f$. Then we have monomial bases of standard
$\hat{\mathfrak g}$-modules constructed in \cite{MP1}, \cite{MP2}
and \cite{FKLMM}:

{\it For integral dominant $\Lambda=k_0\Lambda_0+k_1\Lambda_1$ of
level $k=k_0+k_1$ the set of finite monomial vectors
$$
x(\pi)v_\Lambda=\ \dots\, f(-j)\sp{c_j}h(-j)\sp{b_j}
e(-j)\sp{a_j}\dots\,f(-1)\sp{c_1}h(-1)\sp{b_1}
e(-1)\sp{a_1}f(0)\sp{c_0}v_\Lambda
$$
satisfying difference conditions
$$
\begin{aligned}
c_{j+1}+b_{j+1}+c_{j}&\leq k,\\
b_{j+1}+a_{j+1}+c_{j}&\leq k,\\
a_{j+1}+c_{j}+b_{j}&\leq k,\\
a_{j+1}+b_{j}+a_{j}&\leq k
\end{aligned}
$$
for all $j\geq 0$, and initial conditions $a_1\leq k_0$ and $c_0\leq
k_1$, is a basis of standard $\hat{\mathfrak g}$-module
$L(\Lambda)$.}

These difference and initial conditions for $A_1\sp{(1)}$-module
$L(k\Lambda_0)$ coincide with difference conditions
\eqref{E:difference conditions} and initial conditions
\eqref{E:initial conditions} for $B_2\sp{(1)}$ subspace
$W(k\Lambda_0)$. Moreover, the result of E.~Feigin [F, Theorem 3.1]
implies that $W(k\Lambda_0)$ for $B_2\sp{(1)}$ and $L(k\Lambda_0)$
for $A_1\sp{(1)}$ have the same presentation:

{\it Let $k$ be a positive integer. Let
$$\mathcal P=\mathbb C[f(j),h(j),e(j)\mid
j\leq-1]$$ and let $\mathcal I_{k\Lambda_0}$ be the ideal in the
polynomial algebra $\mathcal P$ generated by polynomials
$$
\bigcup_{\substack{n\leq-k-1}}U(\mathfrak
g)\cdot\Big(\sum_{\substack{j_1,\dots,j_{k+1}\leq-1\\j_1+\dots+j_{k+1}=n}}e(j_1)\dots
e(j_{k+1})\Big)
$$
(here $\cdot$ denotes the adjoint action of $\mathfrak g$ on
$\mathcal P$). Then as $\mathbb Z$-graded vector spaces and
$\mathfrak g$-modules
$$
L(k\Lambda_0)\cong \mathcal P/\mathcal I_{k\Lambda_0}.
$$
}

Due to this coincidence E.~Feigin's fermionic formula [F, Theorem
3.2] for $A_1\sp{(1)}$-module $L(k\Lambda_0)$ is also a character
formula of Feigin-Stoyanovsky type subspace $W(k\Lambda_0)$ for
$B_2\sp{(1)}$.

\end{document}